\documentclass[9pt,reqno]{amsart}

\usepackage{amsmath, amsfonts, amssymb, latexsym, amsthm}
\usepackage{graphicx,subfigure,float,url,color}
\usepackage[pagewise]{lineno}
\usepackage{amsmath, amsfonts, amssymb, latexsym}
\usepackage[numbers,sort&compress]{natbib}
\usepackage{mathrsfs}
\usepackage{pdfcomment}
\usepackage{enumerate}
\hypersetup{hidelinks,
	colorlinks=true,
	allcolors=black,
	pdfstartview=Fit,
	breaklinks=true
}

\everymath{\displaystyle}

\newtheorem{theorem}{Theorem}
\theoremstyle{plain}

\newtheorem{lemma}[theorem]{Lemma}
\newtheorem{definition}[theorem]{Definition}

\newtheorem{proposition}[theorem]{Proposition}
\newtheorem{corollary}[theorem]{Corollary}
\newtheorem{remark}[theorem]{Remark}

\numberwithin{equation}{section}
\numberwithin{theorem}{section}

\newcommand{\cB}{\mathcal{B}}

\newcommand{\cH}{\mathcal{H}}

\newcommand{\R}{\mathbb{R}}

\newcommand{\N}{\mathbb{N}}

\def\crr{\cr\noalign{\vskip2mm}}
\def\dfrac{\displaystyle\frac}

\def\Tv{\mathbb{|\kern-0.1em|\kern-0.1em|}}
\renewcommand{\geq}{\geqslant}
\renewcommand{\leq}{\leqslant}

\def\o{\omega}
\def\O{\Omega}

\begin{document}
\title
{Optimal Actuator Location of the Norm Optimal Controls for Degenerate Parabolic Equations}

\author{\sffamily Yuanhang Liu$^{1}$, Weijia Wu$^{1,*}$, Donghui Yang$^1$   \\
	{\sffamily\small $^1$ School of Mathematics and Statistics, Central South University, Changsha 410083, China. }
}
	\footnotetext[1]{Corresponding author: weijiawu@yeah.net }

\email{liuyuanhang97@163.com}
\email{weijiawu@yeah.net}
\email{donghyang@outlook.com}


\maketitle

\begin{abstract}
This paper focuses on investigating the optimal actuator location for achieving minimum norm controls in the context of approximate controllability for degenerate parabolic equations. We propose a formulation of the optimization problem that encompasses both the actuator location and its associated minimum norm control. Specifically, we transform the problem into a two-person zero-sum game problem, resulting in the development of four equivalent formulations. Finally, we establish the crucial result that the solution to the relaxed optimization problem serves as an optimal actuator location for the classical problem. 
\end{abstract}

\pagestyle{myheadings}
\thispagestyle{plain}
\markboth{Optimal Actuator Location For Degenerate Parabolic Equations}{YUANHANG LIU, WEIJIA WU, AND DONGHUI YANG}

\section{Introduction}
The degenerate parabolic equation, belonging to a widely studied class of diffusion equations, holds significant importance in describing various physical phenomena. Examples of such phenomena include laminar flow, large ice blocks, solar radiation-climate interactions, and population genetics (refer to \cite{cannarsa2016global} for more comprehensive descriptions). Consequently, control problems associated with these equations have attracted substantial attention in the research community. Notably, studies on controllability and observability problems for degenerate parabolic equations have been conducted in \cite{cannarsa2016global} and extensively referenced therein.

In recent years, there has been significant interest in shape optimization problems, also known as shape design problems, for both ordinary differential equations (ODEs) and partial differential equations (PDEs). When dealing with control systems governed by differential equations, shape optimization problems manifest in various forms. Unlike lumped parameter systems, where the location of actuators is fixed, in control systems governed by PDEs, the optimal control aims to optimize performance and often allows for flexibility in choosing the actuator location (\cite{morris2010linear}).  The influence of actuator location on noise reduction performance was demonstrated using a simplified duct model in \cite{morris1998noise}.
To determine the optimal actuator location for abstract infinite-dimensional systems and minimize the cost functional under the worst initial condition, an approximation scheme was proposed in \cite{morris2010linear}.
The actuator location problem has received significant attention from researchers in various contexts, with a focus on one-dimensional PDEs, as previously explored in \cite{allaire2010long,darivandi2013algorithm,hebrard2005spillover,privat2013optimal,privat2015complexity,wouwer2000approach}.
Numerical investigations play a crucial role in this research area, providing valuable insights alongside other perspectives \cite{allaire2010long,muench2008optimal,munch2009optimal,munch2011optimal,tiba2011finite}. In \cite{munch2006optimal}, the author investigated the optimization problem involving the shape and position of the damping set for achieving internal stabilization of a linear two-dimensional wave equation.
In \cite{munch2013numerical}, the authors focused on the numerical approximation of null controls with minimal $L^\infty$-norm for a linear heat equation, taking into account a bounded potential.
An intriguing study was presented in \cite{privat2015complexity}, which dealt with the determination of a measurable subset that maximizes the $L^2$ norm of the restricted solution to a homogeneous wave equation over a bounded, connected subset during a finite time interval. In \cite{guo2015optimal}, the research focused on shape optimization problems concerning the norm optimal and time optimal control of null-controlled heat equations. However, the controlled domains considered in \cite{guo2015optimal} were restricted to a specific class of open subsets measured using the Hausdorff metric. Similar limitations were also observed in the shape optimization problems discussed in \cite{guo2012some,guo2013convergence}.
Recently, the study presented in \cite{privat2015complexity} addressed optimal shape and location problems of sensors for parabolic equations with random initial data. It is worth noting that the study showed in \cite{yang2016observability,yang2017optimal} considered the shape optimization problems for stochastic heat equations.

Thus far, no research has been conducted on the optimization of shape problems for degenerate parabolic equations. This study represents the inaugural attempt to address such problems within the framework of degenerate parabolic equations. Expanding upon the groundwork laid by \cite{xu2023existence}, which examined a shape design problem for approximate controllable heat equations, our objective is to explore a shape design problem for approximate controllability in the context of degenerate parabolic equations. In this particular context, the varying domains correspond to the internal actuator domains that can vary across all measurable subsets, provided they possess a prescribed measure at any given decision-making moment, within certain open bounded domains that are situated away from the degeneracy point. It is important to note that this work diverges from \cite{xu2023existence}. On one hand, the result (Theorem \ref{main-theorem}, to be presented later) can be obtained within the case of degenerate parabolic equations. On the other hand, the equation's degeneracy introduces challenges when studying the properties of analyticity (see \eqref{key-inequality}, to be given later), thus complicating the analysis of shape optimization problems. To overcome this challenge, we employ relevant techniques from \cite{liu2023observability} to establish the requisite properties of analyticity.

The subsequent sections of this manuscript are organized as follows: Section 2 elucidates the principal issues under investigation and outlines the significant findings. In Section 3, essential supplementary findings essential to the analysis are provided. Section 4 delves into the exploration of four different formulations that are equivalent to the problem being studied. Lastly, Section 5 furnishes the rigorous proofs of the main theorem.
\section{Problems formulation and main results}
In this section, we will outline the principal issues investigated and present the pivotal findings of this manuscript. Firstly, let us introduce necessary notations.

Let $T>0$ be a fixed positive time constant, and $\O:=(0,1)$. Throughout this paper, we denote by $\langle\cdot,\cdot\rangle$ the scalar product in $L^2(\O)$ and denote by $\|\cdot\|$ the norm induced by $\langle\cdot,\cdot\rangle$. We denote by $|\cdot|$ the Lebesgue measure on $\R$.

Let $A$ be an unbounded linear operator on $L^2(\O)$:
\begin{equation}\label{operator}
\begin{cases}
\mathcal{D}(A):=\{v\in H_\alpha^1(\O):(x^{\alpha}v_x)_x\in L^2(\O) ~\text{and}~ BC_\alpha(v)=0 \},\\
Av :=  (x^{\alpha}v_x)_x,\ \forall v\in\mathcal{D}(A),\,\alpha\in(0,2),
\end{cases}
\end{equation}
where
$$
\begin{array}{ll}
H^1_\alpha(\O):=\bigg\{v\in L^2(\O):v \text{ is absolutely continuous in}~ \O,
 x^{\frac{\alpha}{2}}v_x\in L^2(\O)\,\text{and}~\,v(1)=0 \bigg\},
\end{array}
$$
and
\begin{equation*}
BC_\alpha(v)=
\begin{cases}
v_{|_{x=0}}, &\alpha\in(0,1),\\
(x^{\alpha}v_x)_{|_{x=0}}, &\alpha\in[1,2),
\end{cases}
\end{equation*}
endowed with the norms
$$
\|v\|^2_{H^1_\alpha(\O)}:=\|v\|^2+\|\sqrt{x^\alpha}v_x\|^2.
$$
By \cite{cannarsa2008carleman}, $A$ is the infinitesimal generator of a strongly continuous semigroup $\{e^{At}\}_{t\geq0}$ on $L^2(\O)$.

Let $\O_1\subset \O$ be defined as $\O_1:=(\epsilon,1)$ for some $\epsilon>0$ and let $0<\tau< T$. Let  $\lambda\in(0,1)$ be a given constant and denote by
 \begin{equation} \label{admissible-area}
{W}=\left\{\omega\subset \Omega_1\bigm| \omega \text{ is measurable
with }|\omega|=\lambda |\Omega_1|\right\}.
\end{equation}

The system, that we consider in this paper, is described by the following degenerate equation:
\begin{equation}
\label{state}
\left\{
\begin{array}{ll}
y_t(x,t)-Ay(x,t)+a(x)y(x,t) = \chi_{(\tau,T)}(t)\chi_\o(x) u(x,t), & \left( x ,t\right) \in \O\times(0,T),   \\[2mm]
y(1,t)=BC_\alpha\big(y(\cdot,t)\big) = 0, & t\in \left(0,T\right), \\[2mm]
y\left(x, 0\right) =y_{0}(x), &  x\in \O\,,
\end{array}
\right.
\end{equation}
where initial state $y_0 \in L^2(\O)$, $y$ is the state variable, the function $a(x)$ is analytic in $\O$, the control $u\in L^2(\O\times(\tau,T))$ and $\omega\in { W}$ is said to be, by abuse of notation, the
actuator location. By \cite{cannarsa2008carleman}, one can check that for all $y_0\in L^2(\O)$, system (\ref{state}) admit a unique solution $y$ in the space $C([0,T];L^2(\O))\cap L^2(0,T;H_\alpha^1(\O))$.

Let $y_d\in L^2(\Omega)$ represent the desired state and $\varepsilon_0>0$ denote an admissible error threshold. Let $\omega\in { W}$ and $y_0\in L^2(\Omega)$ be fixed. We address the following minimal norm control problem, aiming to guide the initial state $y_0(\cdot)$ towards a $\varepsilon_0$-neighborhood of the target state $y_d$:
%
\begin{equation}\label{NP-xi-omega}
 N(y_0,\omega):=\inf\big\{\|u\|_{L^2(\Omega\times(\tau,T))}\,\big|\,\|y(T;y_0,
\omega;u)-y_d\|\leq \varepsilon_0\big\},
\end{equation}
and consider the following problem:
\begin{equation}\label{NP-omega}
N(\omega):=\sup\limits_{\|y_0\|\leq 1}N(y_0,\omega).
%
\end{equation}
Our objective is to determine the optimal location for an actuator that minimizes the aforementioned quantity $N(\omega)$. To accomplish this, we shall address the following problem:
$$
{\bf (OP)}\quad\ \overline N:= \inf\limits_{\omega\in{ W}}N(\omega)=\inf_{\omega\in W}\sup\limits_{\|y_0\|\leq 1}N(y_0,\omega).
$$
In order to solve problem {\bf (OP)}, we need the following assumption {\bf(H)}:
$$
({\bf H})\quad\ \|y(T;y_0,\o
;0)-y_d\|>\varepsilon_0.
$$

Several remarks  are given in order as follows.
\begin{remark}
We direct our attention to region $\O_1$ rather than region $\O$ in (\ref{admissible-area}) due to the degeneracy of the operator $A$ defined in (\ref{operator}) when $x=0$. In this scenario, the desired analytic properties (\ref{key-inequality}) (to be formulated later) cannot be derived. Moreover, if we consider a control region $\o$, including the singular point $0$, it would prevent us from obtaining a solution to problem {\bf (OP)}, as it would lack the necessary analytic properties. Hence, our focus lies on region $\O_1$ that are distant from the singular point $0$.
\end{remark}
\begin{remark}
Regarding the optimization problem $N(\omega)$ (see \eqref{NP-omega}), we analyze the worst-case scenario concerning the initial condition $y_0$. When considering the case of $y_d=0$, it follows from the linearity of \eqref{state} that for any $\mu>0$,
 \begin{equation}\label{21-01}
 	\begin{split}
 &\inf\big\{\|u\|_{L^2(\Omega\times\tau,T))}\,\big|\,\|y(T;\mu y_0,
\omega;u)\|\leq \mu\varepsilon_0\big\}=\mu\inf\big\{\|u\|_{L^2(\Omega\times(\tau,T))}\,\big|\,\|y(T;y_0, \omega;u)\|\leq \varepsilon_0\big\}
\end{split}
\end{equation}
From
which, it is derived  that
$$
\sup\limits_{y_0\in L^2(\Omega)}N(y_0,\omega)=+\infty.
$$
Thus we need impose a constraint on $y_0$.  Assume that
$$
B(\bar{y}_0,r)=\left\{ y_0\in L^2(\O)\,\big|\,\|y_0-\bar{y}_0\|\leq r,\,\bar{y}_0\in L^2(\O),r>0 \}\right..
$$
It follows from  \eqref{21-01} that
 \begin{equation*}
 \begin{split}
 &\inf\big\{\|u\|_{L^2(\Omega\times(\tau,T))}\,\big|\,\|y(T;y_0,
\omega;u)-y_d\|\leq \varepsilon_0\big\}\\
=&r\inf\big\{\|u\|_{L^2(\Omega\times(\tau,T))}\,\big|\,\|y(T;\hat{y}_0,
\omega;u)-\hat y_d\|\leq \varepsilon_0/r\big\}
\end{split}
 \end{equation*}
with
$$
y_0=\bar{y}_0+r\hat{y}_0,\qquad \|\hat{y}_0\|=1\qquad \text{ and }\qquad\hat y_d=\frac{y_d-y(T;\bar{y}_0,\omega;0)}{r},
$$
where $\hat{y}_0\in L^2(\O)$. Therefore, we can limit ourselves with restriction of  $\|y_0\|\leq 1$ without loss of generality.
\end{remark}
\begin{remark}
To ensure the non-triviality of problem {\bf(OP)}, we impose the assumption {\bf(H)}, which guarantees the existence of at least one $y_0\in L^2(\Omega)$ with $|y_0|\leq 1$. When the assumption {\bf(H)} is not satisfied, it becomes evident that the null control is optimal for the problem $N(y_0,\omega)$ (see \eqref{NP-xi-omega}) for any $\omega\in W$. Consequently, in such instances, any $\omega\in W$ serves as a solution to the problem {\bf(OP)}.
\end{remark}
\begin{remark}
In our approach, we propose incorporating the controls primarily within the initial interval $[\tau,T]$ rather than spanning the entire duration of $[0,T]$. This selection may appear somewhat contrived within the given context, but it is necessitated by technical constraints. Specifically, the decision to focus on (\ref{key-inequality}) (to be formulated later) stems from the fact that its inclusion does not introduce any singularities during the process of establishing the analyticity of $G(\eta)$ given in (\ref{definition-of-G}).
\end{remark}
Now we present the main result of this paper.
\begin{theorem} \label{main-theorem}
Let $\lambda\in(0,1)$.
Given the analyticity of $a(\cdot)$ in $\Omega$ and the fulfillment of assumption {\bf(H)}, the problem {\bf(OP)} demonstrates the existence of at least one solution. Moreover, any solution to {\bf(OP)} is necessarily an upper level set of an analytic function.
\end{theorem}

\section{Auxiliary conclusions}
In this section, we present several supplementary findings that will be utilized in subsequent analyses. Initially, recognizing the absence of compactness in ${W}$, it is imperative to extend the feasible set $W$ to a more relaxed set (see \cite{privat2013optimal,privat2015complexity,privat2015optimal}):
\begin{equation}\label{relax-location}
 \cB:=\Bigr\{\beta\in L^\infty(\Omega_1;[0,1])\Bigm| \displaystyle\int_{\Omega_1} \beta(x)dx=\lambda |\Omega_1|\Bigl\}.
\end{equation}
Note that the set $\cB$ is a relaxation of the set $\{\chi_\o\,\big|\,\o\in W\}$.

In accordance with the formulation,   let us discuss the following controlled system:
\begin{equation} \label{stateR}
 \left\{\begin{array}{ll}
  y_t(x, t)-A y(x,t)+a(x)y(x,t)=\chi_{(\tau,T)}(t)\sqrt{\beta(x)} u(x,t)~~&{\mbox{in }}\Omega\times(0,T),\\
  y(1,t)=BC_\alpha\big(y(\cdot,t)\big) = 0 &t\in (0,T),\\
  y(x,0)=y_0(x)&{\mbox{in }}\Omega,
  \end{array}\right.
\end{equation}
where $\beta\in \cB$ and   control $u\in L^2(\Omega\times(\tau,T))$.
The  solution of (\ref{stateR}) is denoted   by $y(\cdot;y_0,\beta; u)$.
This notation is compatible with $y(\cdot;y_0,\omega; u)$ when $\beta=\chi_\omega$.

Let us introduce the following adjoint equation corresponding equation (\ref{state}) (or (\ref{stateR}))
\begin{equation} \label{adjoint-state}
 \left\{\begin{array}{ll}
  \varphi_t(x, t)+A \varphi(x,t)-a(x)\varphi(x,t)=0, &{\mbox{in }}\Omega\times(0,T),\\
  \varphi(1,t)=BC_\alpha\big(\varphi(\cdot,t)\big) = 0, &t\in (0,T),\\
  \varphi(x,T)=\eta,&{\mbox{in }}\Omega,
  \end{array}\right.
\end{equation}
where the solution $\varphi(\cdot;\eta)$ satisfying the following observability inequality (see \cite{liu2023observability}):
\begin{equation}\label{observability inequality}
\|\varphi(0;\eta)\|\leq\displaystyle C\left[\int^{T}_\tau\int_\Omega\chi_\omega
(x)|\varphi(t;\eta)|^2dxdt\right]^{1/2}, \qquad \forall\, \eta\in L^2(\Omega),
\end{equation}
here $C=C(T,\alpha,\o,\O,\mu)$, $\alpha\in(0,2)$ and $\mu\in(0,1)$ is defined as follows
\begin{equation*}\label{u}
\mu=
\left\{
\begin{array}{lll}
\dfrac{3}{4}, &~\text{if}~ \alpha\in(0,2)\setminus\{1\},\\[3mm]
\dfrac{3}{2\gamma}  ~\text{for any}~\gamma\in(0,2),&~\text{if}~ \alpha=1.
\end{array}
\right.
\end{equation*}

\begin{remark}
The presence of $\sqrt{\beta}$ in \eqref{relax-location}, rather than $\beta,$ may initially seem unusual. However, the convexification of the set $W$ defined by \eqref{admissible-area} in the weak star topology of $L^\infty$ precisely corresponds to the set $\mathcal{B}$ defined by \eqref{relax-location}. It is worth noting that $\sqrt{\beta(x)}$ appears in the control operator $B$ of the abstract evolution equation $\dot y(t) = Ay(t) + Bu(t)$ and in the corresponding observability inequality. The square of the term $\|B^*\varphi(\cdot)\|_{L^2(0,T; L^2(\Omega))}$ can be expressed as:
$$
\int^T_0\|B^*\varphi(t)\|^2 dt=\int^T_0\int_\Omega\beta(x)|\varphi(x,t)|^2 dx dt=\int^T_0\int_\Omega\left|\sqrt{\beta(x)}\varphi(x,t)\right|^2 dx dt.
$$
Since this expression is linear with respect to $\beta$, we choose to adopt $\sqrt{\beta}$ in the above formulation.
\end{remark}

Consequently, the problem $N(y_0,\omega)$ (see \eqref{NP-xi-omega}) and {\bf (OP)} can be reformulated as the following relaxation problems, respectively:
\begin{equation}\label{RNP-xi-omega}
 N(y_0,\beta):=\inf\big\{\|u\|_{L^2(\Omega\times(\tau,T))}\,\big|\,\|y(T;y_0,\beta
;u)-y_d\|\leq \varepsilon_0\big\},
\end{equation}
and
$$
{\bf (ROP)}\quad\ N^*:=
\inf\limits_{\beta\in\cB} N(\beta):=\inf\limits_{\beta\in\cB}\sup\limits_{\|y_0\|\leq 1} N(y_0,\beta).
$$

The notations $N(y_0,\beta)$ and $N(\beta)$ correspond to $N(y_0,\omega)$ and $N(\omega)$, respectively, when $\sqrt{\beta}=\chi_\omega$.
The relaxation process ensures the existence of a solution. However, it introduces a limitation: we are uncertain whether there exists a gap between $\overline N$ and $N^*$. If $\overline N\neq N^*$, solving problem {\bf (ROP)} does not guarantee the solution of problem {\bf (OP)}. Fortunately, in our framework, it can be demonstrated that there is no gap, as we will show later, and problem {\bf (OP)} is solvable.

Let us introduce a map $G: L^2(\Omega)\rightarrow L^2(\Omega)$ by
\begin{equation}\label{definition-of-G}
G(\eta)=
\int^{T}_\tau|\varphi(t;\eta)|^2dt,\quad \text{for any }\eta\in L^2(\Omega),
\end{equation}
where $\varphi(t;\eta)$ is the solution to the system \eqref{adjoint-state} with terminal data $\eta\in L^2(\Omega)$. 

Now we give an observability inequality for adjoint equation (\ref{adjoint-state}) corresponding equation (\ref{stateR}). 
\begin{lemma}\label{lemma2}
Let $\lambda\in (0,1)$. Then the adjoint equation (\ref{adjoint-state}) is exactly observable, i.e., there exists a constant $C_\lambda>0$, independent of $\beta$, but possibly depending on $\lambda$ such that for all $\eta\in L^2(\O)$ and $\beta\in \cB$,
\begin{equation}\label{inequality}
 \hspace{30pt} \|\varphi(0;\eta)\|^2\leq C_{\lambda}\bigr\langle \beta, \,G(\eta)\bigl\rangle,
\end{equation}
where $\varphi$ is the solution to \eqref{adjoint-state}.
\end{lemma}
\begin{proof}

For any $\beta\in \cB$, let
\[
\gamma =\dfrac{|\{\beta\ge \sqrt{\lambda/2}\}|}{|\O_1|}.
\]
Since
\begin{align*}
  \lambda|\O_1|
& = \int_{\O_1} \beta dx = \int_{\{\beta\ge \sqrt{\lambda/2}\}} \beta dx +
\int_{\{\beta<\sqrt{\lambda/2}\}} \beta dx\\
& \leq |\{\beta\ge \sqrt{\lambda/2}\}| + \sqrt{{\lambda}/{2}}\cdot |\{\beta< \sqrt{\lambda/2}\}|,
\end{align*}
from above inequality, we have
\[
\lambda|\O_1|\leq\gamma|\O_1|  + \sqrt{{\lambda}/{2}}(1-\gamma)|\O_1|,
\]
and consequently,
\[
\gamma \geq \frac{2\lambda-\sqrt{2\lambda}}{2-\sqrt{2\lambda}}.
\]
Therefore, we obtain
\begin{equation*}
  \label{eq:7.6.2}
  |\{\beta\geq \sqrt{\lambda/2}\}| \ge \frac{2\lambda-\sqrt{2\lambda}}{2-\sqrt{2\lambda}}|\O_1|>0,
\end{equation*}
for all $\beta\in \cB$.

It then follows from inequality (\ref{observability inequality}) with $\o= \{\beta\geq
\sqrt{\lambda/2}\}$ that
\begin{align*}
   \|\varphi(0;\eta)\|^2
& \leq C  \int_\tau^T\int_\O\chi_{\{\beta\geq
\sqrt{\lambda/2}\}}|\varphi(t;\eta)|^2dxdt  
 \leq C  \int_\tau^T\int_\O\chi_{\{\beta\geq
\sqrt{\lambda/2}\}}\frac{\beta}{\sqrt{\lambda/2}}|\varphi(t;\eta)|^2dxdt\\
& = \frac{C}{\sqrt{\lambda/2}}  \int_\tau^T\int_\O\chi_{\{\beta\geq
\sqrt{\lambda/2}\}}\beta|\varphi(t;\eta)|^2dxdt \leq C_\lambda \int_\tau^T\int_\O\beta|\varphi(t;\eta)|^2dxdt\\
&=C_\lambda\left\langle \beta, \int_\tau^T|\varphi(t;\eta)|^2d t\right\rangle,
\end{align*}
which implies the desired estimate (\ref{inequality}).
\end{proof}
\begin{remark}
Thanks to the observability inequality (\ref{inequality}), we can make sure that the set on the right-hand side of (\ref{RNP-xi-omega}) is not empty from Lemma \ref{Le1}, which will be given later.
\end{remark}

In general, it is not easy (or impossible)
to solve the problem $ N(y_0,\beta)$ (see (\ref{RNP-xi-omega})) directly; see \cite{guo2015optimal} for a special class of subdomains.
Instead, let us introduce a functional
\begin{equation*}
  \label{J-1}
  J_{\varepsilon_0}(\eta,\beta)=\frac{1}{2}\int_\tau^T\|\sqrt{\beta}\varphi(t;\eta)\|^2dt+\left\langle y_0,\varphi(0;\eta) \right\rangle-\left\langle y_d,\eta \right\rangle+\varepsilon_0\|\eta\|,
\end{equation*}
and propose the following variational problem
\begin{equation*}
V_{\varepsilon_0}(\beta):=\inf_{\eta\in L^2(\O)}J_{\varepsilon_0}(\eta,\beta).
\end{equation*}
\begin{lemma}\label{Le1}  Let $y_0\in L^2(\Omega)\setminus\{0\}$ and $\beta\in\cB$ be fixed.
Then, we have the following conclusions:
\begin{itemize}
  \item [(1)] There exists $\eta^*\in L^2(\Omega)$ such that $J_{\varepsilon_0}(\cdot,\beta)$ attains its minimum value at $\eta^*$;
  \item [(2)] The control defined by
  \begin{equation*}  \label{Le1-eq1}
u^*=\chi_{(\tau,T)}\sqrt{\beta}\varphi(\cdot;\eta^*),
\end{equation*}
 is the minimal norm optimal control to the problem $ N(y_0,\beta)$ (see (\ref{RNP-xi-omega})), where  $\varphi(\cdot;\eta^*)$ is a solution of adjoint equation (\ref{adjoint-state}) with terminal value $\varphi(T)=\eta^*$. Moreover, we have
   \begin{equation}  \label{dual-relationship}
V_{\varepsilon_0}(\beta)=-\displaystyle\frac{1}{2}N(y_0,\beta)^2.
\end{equation}
\end{itemize}
\end{lemma}

\begin{proof}
At first, we prove the conclusion ($1$). 

One can easily verify that $J_{\varepsilon_0}(\cdot,\beta)$ is continuous and convex. Thus we only need to show the coercivity.

Let $\{\eta_n\}_{n\geq1}\subset L^2(\O)$ such that $\|\eta_n\|\rightarrow\infty$ as $n\rightarrow\infty$, and let
$$
\bar{\eta}_n=\frac{\eta_n}{\|\eta_n\|}
$$
so that $\|\bar{\eta}_n\|=1$. Then
$$
\frac{J_{\varepsilon_0}(\eta_n,\beta)}{\|\eta_n\|}=\frac{1}{2}\|\eta_n\|\int_\tau^T\|\sqrt{\beta}\varphi(t;\bar{\eta}_n)\|^2dt+\left\langle y_0,\varphi(0;\bar{\eta}_n) \right\rangle-\left\langle y_d,\bar{\eta}_n \right\rangle+\varepsilon_0.
$$
Note that
\begin{align*}
|\left\langle y_d,\bar{\eta}_n \right\rangle|\leq\|y_d\|,
\end{align*}
and
\begin{align*}
    |\left\langle y_0,\varphi(0;\bar{\eta}_n) \right\rangle|\leq\|y_0\|\|\varphi(0;\bar{\eta}_n)\|\leq C\|y_0\|\left[\int^{T}_\tau\|\sqrt{\beta}\varphi(t;\bar\eta_n)\|^2dt\right]^{1/2}.
\end{align*}

Case ($i$). If $\liminf\limits_{n\rightarrow\infty}\int_\tau^T\|\sqrt{\beta}\varphi(t;\bar{\eta}_n)\|^2dt>0$, we have
 \begin{equation}  \label{Le1-eq2}
J_{\varepsilon_0}(\eta_n,\beta)\rightarrow\infty ~\text{as}~n\rightarrow\infty,
\end{equation}
which implies the coercivity of $J_{\varepsilon_0}(\cdot,\beta)$.

Case ($ii$). If $\liminf\limits_{n\rightarrow\infty}\int_\tau^T\|\sqrt{\beta}\varphi(t;\bar{\eta}_n)\|^2dt=0$, then it follows from the observability inequality (\ref{inequality}) that $\bar{\eta}_n$ is bounded in $L^2(\O)$. Thus, we can extract a subsequence $\{\bar{\eta}_{n_j}\}\subset\{\bar{\eta}_{n}\}$ such that $\bar{\eta}_{n_j}\rightarrow \bar{\eta}$ weakly in $L^2(\O)$ and $\varphi(\cdot;\bar{\eta}_{n_j})\rightarrow\varphi(\cdot;\bar{\eta})$ weakly in $L^2(0,T;H_\alpha^1(\O))$. Moreover, by lower semi-continuity we obtain
$$
\int_\tau^T\|\sqrt{\beta}\varphi(t;\bar{\eta})\|^2dt\leq\liminf\limits_{n\rightarrow\infty}\int_\tau^T\|\sqrt{\beta}\varphi(t;\bar{\eta}_{n_j})\|^2dt.
$$
Then by the observability inequality (\ref{inequality}) again, we get $\varphi(0;\bar{\eta})=0$, and thus $\bar\eta=0$, and 
\begin{align*}
\liminf\limits_{n\rightarrow\infty}\frac{J_{\varepsilon_0}(\eta_n,\beta)}{\|\eta_n\|}
&\geq\liminf\limits_{n\rightarrow\infty}\big[ \left\langle y_0,\varphi(0;\bar{\eta}_n) \right\rangle-\left\langle y_d,\bar{\eta}_n \right\rangle+\varepsilon_0 \big]\\
&=\left\langle y_0,\varphi(0;\bar{\eta}) \right\rangle-\left\langle y_d,\bar{\eta} \right\rangle+\varepsilon_0=\varepsilon_0,
\end{align*}
which implies (\ref{Le1-eq2}), and so $J_{\varepsilon_0}(\cdot,\beta)$ is coercive.

To sum up, we showed that $J_{\varepsilon_0}(\cdot,\beta)$ is convex, continuous, and coercive, and thus the minimizer of
$J_{\varepsilon_0}(\cdot,\beta)$ exists. The conclusion (1) is true.

Next, we shall prove the conclusion (2) and the rest of the proof will be carried out by the following steps.

{\it Step 1. Show that $u^*=\chi_{(\tau,T)}\sqrt{\beta}\varphi(\cdot;\eta^*)$ is an admissible control for problem $N(y_0,\beta)$.} 

Since $J_{\varepsilon_0}(\cdot,\beta)$ attains its minimum value at $\eta^*$, for any $\eta\in L^2(\O)$ and $h\in\R$, we have
$$
J_{\varepsilon_0}(\eta^*,\beta)\leq J_{\varepsilon_0}(\eta^*+h\eta,\beta).
$$
Thus,
\begin{equation}\label{Le1-eq}
\begin{split}
&J_{\varepsilon_0}(\eta^*+h\eta,\beta)-J_{\varepsilon_0}(\eta^*,\beta)\\
=&\frac{1}{2}\int_\tau^T\|\sqrt{\beta}\varphi(t;\eta^*+h\eta)\|^2dt+\left\langle y_0,\varphi(0;\eta^*+h\eta) \right\rangle-\left\langle y_d,\eta^*+h\eta \right\rangle+\varepsilon_0\|\eta^*+h\eta\|\\
&-\frac{1}{2}\int_\tau^T\|\sqrt{\beta}\varphi(t;\eta^*)\|^2dt+\left\langle y_0,\varphi(0;\eta^*) \right\rangle-\left\langle y_d,\eta^* \right\rangle+\varepsilon_0\|\eta^*\|\\
=&\frac{h^2}{2}\int_\tau^T\|\sqrt{\beta}\varphi(t;\eta)\|^2dt+h\int_\tau^T\left\langle\sqrt{\beta}\varphi(t;\eta^*),\sqrt{\beta}\varphi(t;\eta)
\right\rangle dt+h\left\langle y_0,\varphi(0;\eta) \right\rangle\\
&-h\left\langle y_d,\eta \right\rangle+\varepsilon_0\left( \|\eta^*+h\eta\|-\|\eta^*\| \right)\\
\geq&0.
\end{split}
\end{equation}
Also note that
$$
\big|\|\eta^*+h\eta\|-\|\eta^*\|\big|\leq |h|\|\eta\|,
$$
then we obtain
\begin{equation}\label{Le1-eq3}
\begin{split}
&\frac{h^2}{2}\int_\tau^T\|\sqrt{\beta}\varphi(t;\eta)\|^2dt+h\int_\tau^T\left\langle\sqrt{\beta}\varphi(t;\eta^*),\sqrt{\beta}\varphi(t;\eta)
\right\rangle dt-h\left\langle y_d,\eta \right\rangle\\
+&\varepsilon_0|h|\|\eta\|
+h\left\langle y_0,\varphi(0;\eta) \right\rangle\geq0.
\end{split}
\end{equation}
If $h>0$, then dividing $h$ and sending $h \rightarrow 0^+$ in (\ref{Le1-eq3}) yield
$$
\int_\tau^T\left\langle\sqrt{\beta}\varphi(t;\eta^*),\sqrt{\beta}\varphi(t;\eta)
\right\rangle dt-\left\langle y_d,\eta \right\rangle
+\varepsilon_0\|\eta\|+\left\langle y_0,\varphi(0;\eta) \right\rangle\geq0.
$$
If $h<0$, then dividing $h$ and sending $h \rightarrow 0^-$ in (\ref{Le1-eq3}) yield
$$
\int_\tau^T\left\langle\sqrt{\beta}\varphi(t;\eta^*),\sqrt{\beta}\varphi(t;\eta)
\right\rangle dt-\left\langle y_d,\eta \right\rangle
-\varepsilon_0\|\eta\|+\left\langle y_0,\varphi(0;\eta) \right\rangle\leq0.
$$
Consequently,
\begin{equation}\label{Le1-eq4}
\left|\int_\tau^T\left\langle\sqrt{\beta}\varphi(t;\eta^*),\sqrt{\beta}\varphi(t;\eta)
\right\rangle dt-\left\langle y_d,\eta \right\rangle
+\left\langle y_0,\varphi(0;\eta) \right\rangle\right|\leq\varepsilon_0\|\eta\|.
\end{equation}
Multiplying $\varphi(t;\eta)$ on both sides of the first equation of (\ref{stateR}) and integrating with respect to  $x\in\O$ and $t\in(0,T)$, it produces
\begin{equation}\label{Le1-eq5}
\int_0^T\left\langle u,\sqrt{\beta}\varphi(t;\eta)
\right\rangle dt+\left\langle y_0,\varphi(0;\eta) \right\rangle=\left\langle y(T;y_0,\beta;u),\eta \right\rangle.
\end{equation}
Plugging (\ref{Le1-eq5}) into (\ref{Le1-eq4}) and letting $u=u^*=\chi_{(\tau,T)}\sqrt{\beta}\varphi(\cdot;\eta^*)$, we obtain
$$
\left\langle y(T;y_0,\beta;u^*)-y_d,\eta \right\rangle\leq\varepsilon_0\|\eta\|.
$$
Since $\eta$ is an arbitrary element in $L^2(\O)$, we conclude that
$$
\|y(T;y_0,\beta;u^*)-y_d\|\leq\varepsilon_0.
$$
Therefore, $u^*=\chi_{(\tau,T)}\sqrt{\beta}\varphi(\cdot;\eta^*)$ is an admissible control.

{\it Step 2. Show that $V_{\varepsilon_0}(\beta)<0$.}
  	
   Indeed, since $y_0\neq 0$ and
		\begin{equation*}
		\overline{ \{\varphi(0; \eta)\bigm| \eta\in L^2(\Omega)\}}=L^2(\Omega),
		\end{equation*}
		for every  $\varepsilon\in (0,\frac{\|y_0\|}{2})$  there exists $\eta_\varepsilon\in L^2(\Omega)$ such that
		\begin{equation*}
		\|\eta_\varepsilon\|=\frac{\varepsilon}{2(\|y_d\|+\varepsilon_0)},\,\|\varphi(0; \eta_\varepsilon)\|=1 \mbox{ and } \langle y_0, \varphi(0; \eta_\varepsilon)\rangle<-\|y_0\|+\frac{\varepsilon}{2}. 	
		\end{equation*}
		Let $\xi>0$. Then
		\begin{equation*}
		\begin{split}
		V_{\varepsilon_0}(\beta)
		\leq& \frac{1}{2}\int_\tau^T\|\sqrt{\beta}\varphi(t;\xi\eta_\varepsilon)\|^2dt+\left\langle y_0,\varphi(0;\xi\eta_\varepsilon) \right\rangle-\left\langle y_d,\xi\eta_\varepsilon \right\rangle+\varepsilon_0\|\xi\eta_\varepsilon\|\\
		=&  \frac{\xi^2}{2}\int_\tau^T\|\sqrt{\beta}\varphi(t;\eta_\varepsilon)\|^2dt+\xi\left[-\|y_0\|+\frac{\varepsilon}{2}+(\|y_d\|+\varepsilon_0)\|\eta_\varepsilon\|\right]\\
		=& \frac{\xi^2}{2} \int_\tau^T\|\sqrt{\beta}\varphi(t;\eta_\varepsilon)\|^2dt+\xi(-\|y_0\|+\varepsilon)\crr
		\leq& \frac{\xi^2}{2} \int_\tau^T\|\sqrt{\beta}\varphi(t;\eta_\varepsilon)\|^2dt-\frac{\xi}{2}\|y_0\|.
		\end{split}
		\end{equation*}
		Take $\xi>0$ small enough to obtain
		\begin{equation*}
		V_{\varepsilon_0}(\beta)<0.
		\end{equation*}

{\it Step 3. Show that $u^*=\chi_{(\tau,T)}\sqrt{\beta}\varphi(\cdot;\eta^*)$ is the minimum norm control for problem $N(y_0,\beta)$.}

To this end, let's go back to the inequality (\ref{Le1-eq}). Since $\varphi(\cdot;\eta^*)\neq 0$ by {\it Step 2.}, from the observability inequality (\ref{inequality}) we obtain \begin{equation*}\label{***}
		\varphi(\cdot;\eta^*)\neq 0 \mbox{ for a.e. } t\in [0,T].
		\end{equation*}
Then note that
$$
\lim_{h\rightarrow0}\frac{\|\eta^*+h\eta\|-\|\eta^*\|}{h}=\frac{\left\langle \eta^*,\eta \right\rangle}{\|\eta^*\|}.
$$
If $h>0$, then dividing $h$ and sending $h \rightarrow 0^+$ in (\ref{Le1-eq}) yield
$$
\int_\tau^T\left\langle\sqrt{\beta}\varphi(t;\eta^*),\sqrt{\beta}\varphi(t;\eta)
\right\rangle dt-\left\langle y_d,\eta \right\rangle
+\varepsilon_0\frac{\left\langle \eta^*,\eta \right\rangle}{\|\eta^*\|}+\left\langle y_0,\varphi(0;\eta) \right\rangle\geq0.
$$
If $h<0$, then dividing $h$ and sending $h \rightarrow 0^-$ in (\ref{Le1-eq}) yield
$$
\int_\tau^T\left\langle\sqrt{\beta}\varphi(t;\eta^*),\sqrt{\beta}\varphi(t;\eta)
\right\rangle dt-\left\langle y_d,\eta \right\rangle
+\varepsilon_0\frac{\left\langle \eta^*,\eta \right\rangle}{\|\eta^*\|}+\left\langle y_0,\varphi(0;\eta) \right\rangle\leq0.
$$
To sum up, we have the following Euler-Lagrange equation
\begin{equation}\label{E-L}
\int_\tau^T\left\langle\sqrt{\beta}\varphi(t;\eta^*),\sqrt{\beta}\varphi(t;\eta)
\right\rangle dt-\left\langle y_d,\eta \right\rangle
+\varepsilon_0\frac{\left\langle \eta^*,\eta \right\rangle}{\|\eta^*\|}+\left\langle y_0,\varphi(0;\eta) \right\rangle=0.
\end{equation}
Plugging (\ref{Le1-eq5}) into (\ref{E-L}) and letting $u=u^*=\chi_{(\tau,T)}\sqrt{\beta}\varphi(\cdot;\eta^*)$, we obtain
$$
\left\langle y(T;y_0,\beta;u^*)-y_d,\eta \right\rangle+\varepsilon_0\frac{\left\langle \eta^*,\eta \right\rangle}{\|\eta^*\|}=0.
$$
If particular, by choosing $\eta=\eta^*$,
\begin{equation}\label{Le1-eq6}
\left\langle y(T;y_0,\beta;u^*)-y_d,\eta^* \right\rangle=-\varepsilon_0\|\eta^*\|.
\end{equation}

Suppose now there is another admissible control $\bar{u}$ such that
$$
\|y(T;y_0,\beta;\bar{u})-y_d\|\leq\varepsilon_0.
$$
Then we have by (\ref{Le1-eq6}),
$$
\left\langle y(T;y_0,\beta;\bar{u})-y_d,\eta^*\right\rangle\geq-\|y(T;y_0,\beta;\bar{u})-y_d\|\|\eta^*\|\geq-\varepsilon_0\|\eta^*\|=\left\langle y(T;y_0,\beta;u^*)-y_d,\eta^* \right\rangle.
$$
Using (\ref{Le1-eq5}) again, we arrive at
$$
\int_\tau^T\left\langle\sqrt{\beta}\varphi(t;\eta^*),\bar{u}
\right\rangle dt\geq\int_\tau^T\left\langle\sqrt{\beta}\varphi(t;\eta^*),u^*
\right\rangle dt.
$$
Since $u^*=\chi_{(\tau,T)}\sqrt{\beta}\varphi(\cdot;\eta^*)$, we obtain
$$
\int_\tau^T\|u^*\|^2dt\leq\int_\tau^T\left\langle u^*,\bar{u}
\right\rangle dt,
$$
and by Cauchy-Schwarz inequality
$$
\int_\tau^T\|u^*\|^2dt\leq\int_\tau^T\|\bar{u}\|^2dt,
$$
which implies the optimality of $u^*$.

{\it Step 4. Show that $V_{\varepsilon_0}(\beta)=-\displaystyle\frac{1}{2}N(y_0,\beta)^2$.}

According to Euler-Lagrange equation (\ref{E-L}) and choosing $\eta=\eta^*$, it yields
\begin{equation}\label{Le1-eq7}
\int_\tau^T\|\sqrt{\beta}\varphi(t;\eta^*)\|^2 dt-\left\langle y_d,\eta^* \right\rangle
+\varepsilon_0\|\eta^*\|+\left\langle y_0,\varphi(\tau;\eta^*) \right\rangle=0.
\end{equation}
Setting ${\sqrt{\beta}}\varphi(\cdot;\eta^*)={u^*}$ in (\ref{Le1-eq7}) yield
\begin{equation*}\label{Le1-eq8}
\int_\tau^T\|u^*\|^2 dt=\left\langle y_d,\eta^* \right\rangle
-\varepsilon_0\|\eta^*\|-\left\langle y_0,\varphi(0;\eta^*) \right\rangle,
\end{equation*}
which, along with the definition of $N(y_0,\beta)$ (see (\ref{RNP-xi-omega})), it stands
\begin{equation}\label{Le1-eq9}
N(y_0,\beta)^2=\left\langle y_d,\eta^* \right\rangle
-\varepsilon_0\|\eta^*\|-\left\langle y_0,\varphi(0;\eta^*) \right\rangle.
\end{equation}
On the other hand, we have by (\ref{Le1-eq9})
\begin{align*}
V_{\varepsilon_0}(\beta)=&J_{\varepsilon_0}(\eta^*,\beta)=\frac{1}{2}\int_\tau^T\|\sqrt{\beta}\varphi(t;\eta^*)\|^2dt+\left\langle y_0,\varphi(0;\eta^*) \right\rangle-\left\langle y_d,\eta^* \right\rangle+\varepsilon_0\|\eta^*\|\\
=&\frac{1}{2}\int_\tau^T\|u^*\|^2dt+\left\langle y_0,\varphi(0;\eta^*) \right\rangle-\left\langle y_d,\eta^* \right\rangle+\varepsilon_0\|\eta^*\|=\frac{1}{2}N(y_0,\beta)^2-N(y_0,\beta)^2\\
=&-\frac{1}{2}N(y_0,\beta)^2.
\end{align*}
The proof is completed.
\end{proof}

\section{The equivalent problems}
In this section, we present four equivalent formulations of problem {\bf (ROP)}, namely, problem {\bf (SP1)}, problem {\bf (SP2)}, problem {\bf (SP3)}, and problem {\bf (SP4)}.

\subsection{The first equivalent problem}
In this subsection, we introduce the first  problem {\bf (SP1)}. From Lemma \ref{Le1}, we have
  \begin{equation}\label{1.09}
\begin{split}
     \frac{1}{2}(N^*)^2
    &= \inf\limits_{\beta\in\cB} \displaystyle\frac{1}{2} N(\beta)^2=\inf\limits_{\beta\in\cB}\sup\limits_{\|y_0\|\le 1} \displaystyle\frac{1}{2} N(y_0, \beta)^2\\
    &=\inf\limits_{\beta\in\cB}\sup\limits_{\|y_0\|\le 1}\left[- \inf\limits_{\eta\in L^2(\Omega)}\left(\displaystyle\frac{1}{2}\left\|
\sqrt{\beta}\varphi(\cdot;\eta)\right\|^2_{L^2(\Omega\times(\tau,T))}+\langle y_0,\varphi(0;\eta)\rangle-\langle y_d,\eta\rangle+\varepsilon_0\|\eta\|\right)\right]\\
    &=-\sup\limits_{\beta\in\cB}\inf\limits_{\|y_0\|\le 1} \inf\limits_{\eta\in L^2(\Omega)}\left[ \displaystyle\frac{1}{2}\left\|
\sqrt{\beta}\varphi(\cdot;\eta)\right\|^2_{L^2(\Omega\times(\tau,T))}+\langle y_0,\varphi(0;\eta)\rangle-\langle y_d,\eta\rangle+\varepsilon_0\|\eta\|\right]\\
&=-\sup\limits_{\beta\in\cB} \inf\limits_{\eta\in L^2(\Omega)}\inf\limits_{\|y_0\|\le 1}\left[ \displaystyle\frac{1}{2}\left\|
\sqrt{\beta}\varphi(\cdot;\eta)\right\|^2_{L^2(\Omega\times(\tau,T))}+\langle y_0,\varphi(0;\eta)\rangle-\langle y_d,\eta\rangle+\varepsilon_0\|\eta\|\right]\\
&=-\sup\limits_{\beta\in\cB} \inf\limits_{\eta\in L^2(\Omega)}\left[ \displaystyle\frac{1}{2}\left\|
\sqrt{\beta}\varphi(\cdot;\eta)\right\|^2_{L^2(\Omega\times(\tau,T))}-\|\varphi(0;\eta)\|-\langle y_d,\eta\rangle+\varepsilon_0\|\eta\|\right],
        \end{split}
\end{equation}
where  \eqref{dual-relationship} was  used in the third equality.
Define a cost functional $J(\cdot,\cdot)$ by
\begin{equation}\label{cost-function}
J(\beta,\eta)=\displaystyle\frac{1}{2}\left\langle
\beta,\,G(\eta)\right\rangle-\|\varphi(0;\eta)\|-\langle y_d,\eta\rangle+\varepsilon_0\|\eta\|
\quad \text{for any }\beta\in \cB,\, \eta\in  L^2(\Omega),
\end{equation}
where $G(\eta)$ is defined in (\ref{definition-of-G}).

By utilizing the expression in \eqref{1.09}, we can reformulate problem {\bf (ROP)} as a two-level optimization problem, commonly referred to as a Stackelberg game problem.
\begin{equation*}
{\bf (SP1)}\quad\ \sup\limits_{\beta\in{\cB}}\inf\limits_{\eta\in L^2(\Omega)}
J(\beta,\eta).
\end{equation*}

\begin{remark}\label{remark1-2}
The equivalence between problem {\bf (ROP)} and problem {\bf (SP1)} is evident. For any given $\beta\in\cB$, the continuity of $G$ implies that
\begin{equation}\label{continuity-of-J}
J(\beta,\cdot)\in C(L^2(\Omega)).
\end{equation}
However, the term $-\|\varphi(\tau;\eta)\|$ in \eqref{cost-function} is not convex with respect to $\eta$, leading to the non-convexity of $J(\beta,\cdot)$ for any $\beta\in\cB$. This constitutes the primary challenge that we must overcome.
\end{remark}
\subsection{The second equivalent problem}
In this subsection, we will present the second equivalent problem, denoted as {\bf (SP2)}, which will be given later. For a fixed $\beta\in\cB$, we will focus on the first-level optimization of problem {\bf (SP1)}. i.e.
\begin{equation*}\label{1-18}
\inf\limits_{\eta\in L^2(\Omega)}J(\beta,\eta).
\end{equation*}
In order to ensure compactness, we will confine the domain of $J(\beta,\cdot)$ within a closed sphere centered at the origin. To facilitate this, we introduce some notations. Due to the density of the range of the operator $e^{A T}$ in $L^2(\Omega)$ (as shown in \cite{cannarsa2008carleman,alabau2006carleman}), there exists $\hat y_0\in L^2(\Omega)$ such that
\begin{equation}\label{definition-of-y1}
\left\|e^{A T}\hat y_0-y_d\right\|< \frac{\varepsilon_0}{2}.
\end{equation}
 Let  \begin{equation}\label{definition-of-delta}
 \delta_0=\frac{C_\lambda}{\varepsilon_0}(1+\|\hat y_0\|)^2,
 \end{equation}
 where $C_\lambda$ is the constant in (\ref{inequality}). 
 Hereafter, we denote
 \begin{equation}\label{definition of B}
\hat{\cB}=\left\{\eta\in L^2(\Omega)\bigm| \|\eta\|\leq\delta_0\right\}.
 \end{equation}
Besides,
 let
\begin{equation}\label{1.15}
D_\beta=\left\{\eta\in L^2(\Omega)\bigm|J(\beta,\eta)\leq0\right\},
\end{equation}
where $J(\cdot,\cdot)$ is defined in  \eqref{cost-function}.
The non-emptiness of $D_\beta$ is evident as $0$ belongs to $D_\beta$. Now, we present the following result pertaining to the relationship between $\hat{\cB}$ and $D_\beta$.

\begin{lemma}\label{lemma3}
Let $\beta\in{\cB}$ be given, and let $\hat{\cB}$ and $D_\beta$ be   defined by \eqref{definition of B} and \eqref{1.15}, respectively. Then,
\begin{equation*}
D_\beta\subseteq \hat{\cB} \text{ for any }\beta\in\cB.
\end{equation*}
 \end{lemma}

  \begin{proof}
Rewrite $J(\beta,\eta)$  from \eqref{cost-function} as
\begin{equation}\label{1.17}
 \begin{split}
 J(\beta,\eta)=&\displaystyle\frac{1}{2}\left\langle
\beta,\,G(\eta)\right\rangle-\|\varphi(0;\eta)\|-\langle y_d,\eta\rangle+\varepsilon_0\|\eta\|
 \\=& \displaystyle\frac{1}{2}\left\langle
\beta,\,G(\eta)\right\rangle-\|\varphi(0;\eta)\|-\langle\hat y_0,\varphi(0;\eta)\rangle-\langle y_1,\eta\rangle+\tau\varepsilon_0\|\eta\|,
 \end{split}
\end{equation}
where
$$
y_1:= y_d-e^{A T}\hat y_0.
$$
By  \eqref{definition-of-y1}, it follows  that
\begin{equation}\label{1.18}
\|y_1\|<\frac{\varepsilon_0}{2}.
\end{equation}

Now assume that
 $\eta\in D_\beta$.
Then,  it follows from \eqref{1.15} and \eqref{1.17}  that
\begin{equation}\label{bz2}
\varepsilon_0\|\eta\|-\langle y_1, \eta\rangle
   \leq\displaystyle -\frac{1}{2}\langle \beta, G(\eta)\rangle+(1+\|\hat y_0\|)\cdot\|\varphi(0;\eta)\|.
\end{equation}
  On the other hand, by  \eqref{1.18},
\begin{equation}\label{bz3}
\displaystyle\frac{\varepsilon_0}{2}\|\eta\|\leq
    \varepsilon_0\|\eta\|-\langle y_1, \eta\rangle.
\end{equation}
The above two inequalites (\ref{bz2}) and (\ref{bz3}), together with
 the observation inequality \eqref{inequality}, lead to
 \begin{equation*}
 	\begin{split}
 \displaystyle\frac{\varepsilon_0}{2}\|\eta\|&\leq
  -\frac{1}{2}\langle \beta, G(\eta)\rangle+(1+\|\hat y_0\|)\cdot\|\varphi(0;\eta)\|
 \leq \displaystyle \left[-\frac{1}{2C_\lambda}\|\varphi(0;\eta)||^2+(1+\|\hat y_0\|)\cdot\|\varphi(0; \eta)\|\right]
 \\
 &\leq\sup\limits_{\hat \eta\in L^2(\Omega)}\displaystyle \left[-\frac{1}{2C_\lambda}\|\varphi(0;\hat \eta)||^2+(1+\|\hat y_0\|)\cdot\|\varphi(0;\hat \eta)\|\right]
 \leq\displaystyle\frac{C_\lambda}{2}(1+\|\hat y_0\|)^2.
 \end{split}
 \end{equation*}
 This, together with \eqref{definition-of-delta} and \eqref{definition of B}, implies that
  $\eta\in \hat{\cB}$. Therefore, $D_\beta\subseteq\hat{\cB}$. The proof is completed.
\end{proof}

Now, we present the second equivalent problem. Let $\beta\in{\cB}$.
Since $$\inf\limits_{\eta\in L^2(\Omega)}J(\beta, \eta)\leq J(\beta, 0)=0$$ and  $D_\beta\subseteq\hat{\cB}$ from Lemma \ref{lemma3},
it is derived that
$$
\inf\limits_{\eta\in L^2(\Omega)}J(\beta,\eta)=\inf\limits_{\eta\in D_\beta}J(\beta, \eta)=\inf\limits_{\eta\in \hat{\cB}}J(\beta, \eta).
$$
Therefore, problem  {\bf (ROP)} is also equivalent to the following problem
\begin{equation*}
{\bf (SP2)}\quad\ \sup\limits_{\beta\in\cB}\inf\limits_{\eta\in \hat{\cB}}J(\beta, \eta).
\end{equation*}
In addition, we have the following estimate from Lemma \ref{lemma2} and \eqref{definition of B} that
\begin{equation}\label{estimate-25}
0=J(\beta, 0)\geq \inf\limits_{\eta\in \hat{\cB}}J(\beta, \eta)
\geq  \inf\limits_{\eta\in \hat{\cB}}\left[\frac{1}{2C_\lambda}\|\varphi(0;\eta)\|^2-\|\varphi(0;\eta)\|-\|y_d\|\|\eta\|\right]\geq -\frac{C_\lambda}{2}-\|y_d\|\delta_0,
\end{equation}
which means that  $\inf\limits_{\eta\in \hat{\cB}}J(\beta, \eta)$ is uniformly bounded with respect to $\beta\in\cB$.

Next, we fix $\beta\in\cB$ and focus on the first-level optimization of problem {\bf (SP2)}. i.e.,
\begin{equation}\label{1-18-1}
\inf\limits_{\eta\in  \hat{\cB}}J(\beta,\eta).
\end{equation}
The following result shows the existence of solutions to problem \eqref{1-18-1} and
{\bf (SP2)}, which will be used in the next subsection.

\begin{lemma}\label{lemma4}
For any $\beta\in{\cB}$, problem \eqref{1-18-1}
 admits at least one solution. Furthermore,
{\bf (SP2)} admits at least one  solution.
 \end{lemma}

 \begin{proof} 
 We divide the proof into two steps.
 
 {\it Step 1. Problem \eqref{1-18-1}
 admits at least one solution.} 
 
 Given the boundedness of $\inf\limits_{\eta\in \hat{\cB}}J(\beta, \eta)$ (as seen in \eqref{estimate-25}), we can find a minimizing sequence $\{\eta_n\}_{n\geq1}\subseteq\hat{\cB}$ for problem \eqref{1-18-1}.
By extracting a subsequence from ${\eta_n}$, which we will also denote as ${\eta_n}$ for simplicity, we can identify an element $\hat \eta\in L^2(\Omega)$ such that
$$
\eta_n\rightarrow  \hat \eta\quad \text{weakly in }L^2(\Omega).
$$
We claim that  $\hat \eta(\cdot)$ solves problem \eqref{1-18-1}. In fact,
$$
\lim\limits_{n\rightarrow\infty}\langle y_d, \eta_n\rangle=\langle y_d, \hat \eta\rangle
 \text{ and }
 \varliminf\limits_{n\rightarrow\infty}\| \eta_n\|\geq \|\hat \eta\|,
 $$
and hence $\hat \eta\in\hat{\cB}$. Furthermore, since the embedding $H_\alpha^1(\O)\hookrightarrow L^2(\O)$ is compact, see \cite{sun2022fundamental}, also see section 6 in \cite{alabau2006carleman}, we have
$$
\lim\limits_{n\rightarrow\infty}\varphi(\cdot; \eta_n)=\varphi(\cdot;\hat \eta)\quad\text{strongly in } L^2(0,T;L^2(\Omega))
$$
 and
 $$
\lim\limits_{n\rightarrow\infty}\varphi(0; \eta_n)=\varphi(0;\hat \eta)\quad\text{strongly in } L^2(\Omega).
$$
Based on the information provided and the definition of $J(\cdot,\cdot)$ in \eqref{cost-function}, we can conclude that $$\lim\limits_{n\rightarrow\infty}J(\beta, \eta_n)\geq J(\beta, \hat \eta).$$ In simpler terms, it means that \eqref{1-18-1} has at least one solution, denoted by $\hat \eta$.

{\it Step 2. {\bf (SP2)} admits at least one  solution.} 

By virtue of the uniform boundedness property (refer to \eqref{estimate-25}), we can identify a maximizing sequence $\{\beta_n\}_{n\geq1}\subseteq\cB$ such that
\begin{equation}\label{1.21}
\lim\limits_{n\rightarrow\infty}\inf\limits_{\eta \in\hat{\cB}} J(\beta_n,\eta)=\sup\limits_{\beta\in\cB}\inf\limits_{\eta \in\hat{\cB}}J(\beta,\eta).
\end{equation}
Once again, we select a subsequence, denoted as the original sequence itself, satisfying
 $$
 \lim\limits_{n\rightarrow\infty}\beta_n=\hat \beta\quad\text{weakly star in } L^\infty(\Omega).
 $$
It can be readily verified that $\hat\beta\in\cB$ and
\begin{equation}\label{1.22}
J(\hat\beta,\eta)=\lim\limits_{n\rightarrow\infty}J(\beta_n, \eta)\quad\text{for any }\eta\in L^2(\Omega).
\end{equation}
Furthermore, due to the existence of a solution to \eqref{1-18-1}, we can identify $\eta_{\hat \beta}\in \hat{\cB}$ such that
$$
J(\hat\beta,\eta_{\hat \beta})=\inf\limits_{\eta\in \hat{\cB}} J(\hat\beta,\eta).
$$
By substituting $\eta=\eta_{\hat \beta}$ into \eqref{1.22}, we deduce from the aforementioned equality that
$$
\inf\limits_{\eta\in \hat{\cB}} J(\hat\beta,\eta)=J(\hat \beta, \eta_{\hat \beta})=\lim\limits_{n\rightarrow\infty}J(\beta_n, \eta_{\hat \beta})\geq \lim\limits_{n\rightarrow\infty}\inf\limits_{\eta\in \hat{\cB}} J(\beta_n,\eta),
$$
which, together with \eqref{1.21}, shows that
\begin{equation*}
	\sup\limits_{\beta\in\cB}\inf\limits_{\eta \in\hat{\cB}}J(\beta,\eta)\leq\inf\limits_{\eta\in \hat{\cB}} J(\hat\beta,\eta)\leq \sup\limits_{\beta\in\cB}\inf\limits_{\eta \in\hat{\cB}}J(\beta,\eta),
\end{equation*}
indicating that $\hat\beta$ represents a solution to {\bf (SP2)}. The proof is completed.
\end{proof}

\subsection{The third equivalent problem}
In this subsection, we shall present the third equivalent problem, denoted as {\bf (SP3)}, which will be given later. 
Due to the non-convexity of $J(\beta,\cdot)$, we need to consider a relaxed version of problem \eqref{1-18-1}.
For this purpose, let $C_b(\hat{\cB})$ denote the space of bounded continuous functions defined on $\hat{\cB}$, equipped with the norm
$$
\|f\|_{C_b(\hat{\cB})}=\sup\limits_{\eta\in\hat{\cB}} |f(\eta)| \text{ for any }f\in C_b(\hat{\cB}).
$$
The dual space of $C_b(\hat{\cB})$ is referred to as $C_b(\hat{\cB})^*$. Furthermore, we denote the dual pairing between $C_b(\hat{\cB})^*$ and $C_b(\hat{\cB})$ as $\langle h, f\rangle_{C_b(\hat{\cB})^*,C_b(\hat{\cB})}$ or simply $h(f)$. More importantly, we define a subset of $C_b(\hat{\cB})^*$, denoted as $\cH(\hat{\cB})$, in the following manner:
\begin{equation}\label{bz4}
h\in \cH(\hat{\cB}) \hbox{ whenever  the following conditions (H1)-(H3) hold}:
\end{equation}
\begin{enumerate}[(H1)]
\item $h\geq0$, it means that
\begin{equation}\label{nonnegative}
\langle h,\, f\rangle_{C_b(\hat{\cB})^*,C_b(\hat{\cB})}\geq 0 \quad\text{when } f(\eta)\geq0,\,\forall\, \eta\in \hat{\cB};
\end{equation}
\item
\begin{equation}\label{norm}
\|h\|_{C_b(\hat{\cB})^*}= 1;
\end{equation}
\item
\begin{equation}\label{attained}
\langle h,\, 1\rangle_{C_b(\hat{\cB})^*,C_b(\hat{\cB})}= 1.
\end{equation}
\end{enumerate}

\begin{remark}
It is worth noting that $\hat{\cB}$ is not a compact set. For any $\eta\in\hat{\cB}$, we define the Dirac function $\delta_{\eta}$ as follows:
$$
\langle \delta_{\eta}, f\rangle_{C_b(\hat{\cB})^*,C_b(\hat{\cB})}=f(\eta) \text{ for any }f\in C_b(\hat{\cB}).
$$
A straightforward verification confirms that
\begin{equation*}
\delta_\eta\in \cH(\hat{\cB}).
\end{equation*}
Hence, we can consider $\cH(\hat{\cB})$ as a relaxation of $\hat{\cB}$.
\end{remark}

From \eqref{continuity-of-J} and \eqref{estimate-25}, we know that
$
J(\beta,\cdot)\in C_b(\hat{\cB})
$
for any $\beta\in \cB$. Besides,  define a cost functional $\widetilde J:\cB\times  \cH(\hat{\cB})\mapsto\mathbb{R}$ by
\begin{equation}\label{cost-function1}
\widetilde J(\beta,h):= \langle h, J(\beta,\cdot)\rangle _{C_b(\hat{\cB})^*,C_b(\hat{\cB})}:=h \bigr(J(\beta,\cdot)\bigl)
\quad \text{for any }\beta\in \cB,\, h\in \cH(\hat{\cB}).
\end{equation}
Now  we propose the following problem
$$
\inf\limits_{h\in \cH(\hat{\cB})}\widetilde J(\beta,h).
$$
Notice that
\begin{equation}\label{1-23}
\widetilde J(\beta,\delta_\eta)=J(\beta,\eta) \quad\text{for any }\beta\in\cB,\, \eta\in\hat{\cB}.
\end{equation}
Thus we could  regard the above problem as a relaxed version of \eqref{1-18-1}. Furthermore, we introduce the third problem:
\begin{equation*}
{\bf (SP3)}\quad\ V^-:=\sup\limits_{\beta\in\cB}\inf\limits_{h\in \cH(\hat{\cB})}\widetilde J(\beta,h).
\end{equation*}

As stated in previous  discussions, we have already proved that the following three expressions are equivalent:
\begin{equation}\label{equivalence}
{\bf (ROP)}\Longleftrightarrow {\bf (SP1)}\Longleftrightarrow{\bf (SP2)}.
\end{equation}
Now we are going to show the following equivalence
\begin{equation}\label{equivalence-1}
{\bf (SP2)}\Longleftrightarrow{\bf (SP3)}.
\end{equation}

In order to establish the equivalence (\ref{equivalence-1}), we rely on the following Lemma \ref{Le5}.

\begin{lemma}\label{Le5}  Let $\beta\in \cB$. Then, there holds
\begin{equation*}
\inf\limits_{\eta\in\hat{\cB}} J(\beta,\eta)=\inf\limits_{h\in \cH(\hat{\cB})}\widetilde J(\beta,h).
\end{equation*}
\end{lemma}
\begin{proof}
By \eqref{1-23}, we have
\begin{equation*}\label{1.30}
\inf\limits_{\eta\in\hat{\cB}} J(\beta,\eta)\geq \inf\limits_{h\in \cH(\hat{\cB})}\widetilde J(\beta,h).
\end{equation*}
On the other hand,  it follows from Lemma \ref{lemma4} that there is $\eta_\beta\in\hat{\cB}$ such that
$$
J(\beta,\eta_\beta)=\inf\limits_{\eta\in\hat{\cB}} J(\beta,\eta),
$$
i.e,  $J(\beta,\cdot)-J(\beta,\eta_\beta)\geq 0$ on $\cB$. This together with the properties  of $h\in \cH(\hat{\cB})$ (see \eqref{nonnegative} and \eqref{attained}), implies  that
$$
h\bigr(J(\beta,\cdot)\bigl)\geq h\bigr(J(\beta,\eta_\beta)\bigl)=J(\beta,\eta_\beta)h(1)=J(\beta,\eta_\beta),
$$
which along with (\ref{cost-function1}), we thus arrive at
$$
\inf\limits_{h\in \cH(\hat{\cB})}\widetilde J(\beta,h)\geq J(\beta,\eta_\beta)=\inf\limits_{\eta\in\hat{\cB}} J(\beta,\eta).
$$
This completes the  proof.
\end{proof}
Utilizing the aforementioned Lemma \ref{Le5}, we can establish the validity of the equivalence (\ref{equivalence-1}). Consequently, the four problems enumerated in \eqref{equivalence} are demonstrated to be equivalent. Furthermore, as deduced from Lemma \ref{lemma4}, we have the following conclusion.
\begin{corollary} \label{remark1} 
The problem {\bf (ROP)} (or {\bf (SP3)}) possesses at least one solution, and $\beta^*\in\cB$ serves as a solution to {\bf (ROP)} if and only if it satisfies {\bf (SP3)}.
\end{corollary}

\subsection{The fourth equivalent problem}
In this subsection, we shall present the fourth equivalent problem denoted as {\bf (SP4)}, which will be given later. To facilitate our discussion, we will employ fundamental definitions and concepts from the field of two-person zero-sum game theory. For a comprehensive understanding of these concepts, we refer the readers to references such as p.121 in \cite{aubin2002optima} or section 3 in \cite{guo2015optimal}.

\begin{definition} \label{Nash equilibrium} Suppose that $ E$ and $F$ are  two sets and  $f:\,E\times  F\mapsto
\mathbb{R}$ is a functional. We call $(\bar y,\bar x)\in
 E\times    F$  a Nash equilibrium if,
$$
f(\bar y, x)\leq f(\bar y,\bar x)\leq f(y,\bar x), \quad\forall\;
y\in E, x\in   F.
$$
\end{definition}

The following result, which can be found in Proposition 8.1 of \cite{aubin2002optima}, establishes a well-known connection between the Stackelberg equilibrium and the Nash equilibrium.

 \begin{proposition} \label{proposition1} The following
conditions are equivalent.
\begin{enumerate}[(i)]
\item $(\bar y,\bar x)$ is a Nash equilibrium;
\item $V^+=V^-$ and  $\bar y$ satisfies the following equation:
$$
V^+:=\inf_{y\in   E}\sup_{x\in  F} f(y,x)=\sup\limits_{x\in  F} f(\bar y,x),
$$
 and $\bar x$ satisfies the following equation:
$$
V^-:=\sup_{x\in   F}\inf_{y\in  E} f(y,x)= \inf\limits_{y\in   E} f(y,\bar x).
$$
\end{enumerate}
When $V^+=V^-$, the common value $V=V^+=V^-$  is called the value of the game.
\end{proposition}

In the above framework of  two-person zero-sum
game problem,  take
\begin{equation*}
E=\cH(\hat{\cB}), \quad F=\cB,  \quad x=h, \quad y=\beta,\quad f(y,x)=\widetilde J(\beta,h).
\end{equation*}
Now  we  introduce the fourth problem obtained  by exchanging the order between ``sup" and ``inf" of {\bf (SP3)}. More precisely,
\begin{equation*}
 {\bf (SP4)}\quad\ V^+=\inf\limits_{h\in \cH(\hat{\cB})}\sup\limits_{\beta\in\cB}\widetilde J(\beta,h).
 \end{equation*}
Then we have the following result.
\begin{proposition}\label{Neumann}
Let
$E$, $F$,  and $f$ be defined as the above.
Then $V^+=V^-$. Besides,  {\bf (SP4)} admits at least one solution. If
$\beta^*\in\cB$  solves {\bf (ROP)} and $h^*\in\cH(\hat{\cB})$ solves {\bf (SP4)},
then
\begin{equation}\label{1-31}
\widetilde J(\beta^*,h^*) =\sup\limits_{\beta\in\cB}\widetilde J(\beta,h^*).
\end{equation}
\end{proposition}
In order to prove Proposition \ref{Neumann}, we need the following results.
\begin{lemma}\label{lemma7}
Let  $\cH(\hat{\cB})$ be defined as in (\ref{bz4}). Then,  $\cH(\hat{\cB})$ is weakly star compact in $C_b(\hat{\cB})^*$.
\end{lemma}

\begin{proof}
We claim that  $\cH(\hat{\cB})$ is weakly star  closed. In fact, assume that
there is $\{h_n\}\subseteq\cH(\hat{\cB})$ so that
$$
h_n\rightarrow \bar h \text{ weakly star in }  C_b(\hat{\cB})^*.
$$
It suffices to prove that $\bar h \in \cH(\hat{\cB})$.
Conditions \eqref{nonnegative} and \eqref{attained} are easy to be  verified for  $\bar h$. It remains to verify \eqref{norm}. Notice that
$$
\|\bar h\|_{C_b(\hat{\cB})^*}=\sup\limits_{\|f\|\leq 1}\langle \bar h, f\rangle
=\sup\limits_{\|f\|\leq 1}\lim\limits_{n\rightarrow\infty}\langle  h_n, f\rangle
\leq \sup\limits_{\|f\|\leq 1}\varlimsup\limits_{n\rightarrow\infty}\|h_n\|\| f\|\leq  \sup\limits_{\|f\|\leq 1}\|f\|=1.
$$
On the other hand, by \eqref{attained}, we have
$$
1=\lim_{n\rightarrow\infty}\langle h_n, 1\rangle_{C_b(\hat{\cB})^*, C_b(\hat{\cB})}=\langle \bar h,\, 1\rangle_{C_b(\hat{\cB})^*,C_b(\hat{\cB})}\leq \|\bar h\|_{C_b(\hat{\cB})^*}\| 1\|_{C_b(\hat{\cB})}=\|\bar h\|_{C_b(\hat{\cB})^*},
$$
which shows that \eqref{norm} holds. 

On the other hand, it follows from the  Banach-Alaoglu-Bourbaki theorem that  the closed unit ball is  weakly star compact in $C_b(\hat{\cB})^*$. Therefore, $\cH(\hat{\cB})$ is weakly star compact in $C_b(\hat{\cB})^*$. The proof is completed.
\end{proof}

\begin{lemma}[{\cite{zeidler2013nonlinear}, Proposition 9.9}]\label{fix}
A mapping $T:K\rightarrow2^K$, where $K\subseteq X$, will have a fixed point if the following conditions are satisfied:
\begin{itemize}
  \item [(1)] $X$ is a locally convex space and the set $K$ is nonempty, compact and convex;
  \item [(2)] The set $T(x)$ is nonempty and convex for all $x\in K$, and the preimages $T^{-1}(y)$ are relatively open with respect to $K$ for all $y\in K$.
\end{itemize}
\end{lemma}

Now, we are ready to prove Proposition \ref{Neumann}.

\begin{proof}[{Proof of Proposition \ref{Neumann}}] We divide the proof into two steps.

{\it Step 1.} The inequality $V^-\leq V^+$ follows immediately from
$$
\inf\limits_{y\in E} f(y,x)\leq \inf\limits_{y\in E}\sup\limits_{\hat x\in F} f(y,\hat x)   \text{ for any }x\in F.
$$

On the contrary, we will demonstrate the inequality $V^- \geq V^+$ using Lemma \ref{fix}.
To accomplish this, we select a fixed value of $\varepsilon > 0$ and define a set-valued mapping $T:\cB\times \cH(\hat{\cB})\mapsto 2^{\cH\times\cH(\hat{\cB})}$ as follows:
\begin{equation*}
T(\beta,h)=\left\{(\hat\beta,\hat h)\in \cB\times \cH(\hat{\cB})\Bigm|\widetilde J(\beta,\hat h)<V^-+\varepsilon,~\widetilde J( \hat\beta,h)>V^+-\varepsilon\right\}.
\end{equation*}

(1) It is evident from the definitions of $V^+$ and $V^-$ that $T(\beta, h)\neq\emptyset$.

(2) By examining the definitions of $\widetilde J(\cdot,\cdot)$ in \eqref{cost-function1} and $J(\cdot,\cdot)$ in \eqref{cost-function}, we observe that $\widetilde J(\beta,\cdot)$ is a linear function of $h\in \cH(\hat{\cB})$, and $\widetilde J(\cdot, h)$ is convex with respect to $\beta\in\cB$. Consequently, $T(\beta, )$ is a convex set.

(3) It is worth noting that $\cB$ is both convex and compact in $L^\infty(\Omega)$ under the weak star topology. Similarly, from Lemma \ref{lemma7}, we ascertain that $\cH(\hat{\cB})$ is convex and compact in $C_b(\hat{\cB})^*$ under the weak star topology.

(4) Notice that
 \begin{equation*}
 \begin{split}
 T^{-1}(\hat\beta,\hat h)&=\left\{(\beta,h)\in \cB\times \cH(\hat{\cB})\Bigm|\widetilde J(\beta,\hat h)<V^-+\varepsilon,~\widetilde J( \hat\beta,h)>V^+-\varepsilon\right\}\\
 &=\left\{\beta\in \cB\Bigm|\widetilde J(\beta,\hat h)<V^-+\varepsilon\right\}\times
 \left\{h\in \cH(\hat{\cB})\Bigm|\widetilde J(\hat\beta,h)>V^+-\varepsilon\right\}.
 \end{split}
 \end{equation*}
Because of the weakly star continuity of $\widetilde J(\beta,\cdot)$ and  $\widetilde J(\cdot, h)$, the  sets
$$
\left\{\beta\in \cB\Bigm|\widetilde J(\beta,\hat h)<V^-+\varepsilon\right\},
\quad
\left\{h\in \cH(\hat{\cB})\Bigm|\widetilde J(\hat\beta,h)>V^+-\varepsilon\right\}
$$
are weakly star open with respect to $\cB$ and $\cH(\hat{\cB})$, respectively. Therefore, $T^{-1}(\hat\beta,\hat h)$ is weakly star relatively open in
$\cB\times \cH(\hat{\cB})$.

Thanks to (1)-(4) and from the  Lemma \ref{fix}, there is a fixed point of $T$, \text{i.e.}, there exists $(\beta,h)\in \cB\times \cH(\hat{\cB})$ with $(\beta,h)\in T(\beta,h)$. Thus
$$
V^+-\varepsilon<\widetilde J(\beta,h)<V^-+\varepsilon.
$$
Since $\varepsilon$ is arbitrary, $V^+\leq V^-$. Thus $V^+= V^-$.

{\it Step 2.} Due to the weak star compactness of $\cH(\hat{\cB})$ in $C_b(\hat{\cB})^*$, as established in Lemma \ref{lemma7}, the weak star continuity of $\widetilde{J}$ implies that at least one solution $h^*\in\cH(\hat{\cB})$ exists for problem {\bf (SP4)} (refer to Proposition 8.2 of \cite{aubin2002optima}).

Moreover, when $\beta^*$ satisfies {\bf (SP2)} according to Lemma \ref{lemma4}, it also satisfies {\bf (SP3)} due to their equivalence (see \eqref{equivalence}). By applying Proposition \ref{proposition1}, we can ascertain that $(h^*,\beta^*)$ constitutes a Nash equilibrium, thereby establishing the validity of the conclusion \eqref{1-31}. This completes the proof.
\end{proof}

Based on the above results, we are able to  prove Theorem \ref{main-theorem} by  \eqref{1-31}  in  the next section.

    \section{Proof of Theorem \ref{main-theorem}}
  In order to establish Theorem \ref{main-theorem}, we require several preliminary results. Our initial objective is to quantify the extent of real analyticity exhibited by the function $G(\cdot)$, as defined in (\ref{definition-of-G}).
\begin{lemma}\label{lemma6}
Let $\alpha\in(0,2)$, $\tau\in(0,T)$ and the operator $G$ be defined by \eqref{definition-of-G}. Then, the  following assertions hold:
\begin{enumerate}[(i)]
\item There are $C_1=C_1(\alpha,\O,\O_1,T,a,\tau)\geq1$, $\rho=\rho(\alpha,\O,\O_1)\in(0,1]$ such  that
\begin{equation}\label{key-inequality}
\left\vert\bigr  (\partial ^\gamma_x G(\eta)\bigl)(x)\right\vert\leq TC_1^2\rho^{-\gamma}\gamma!  \|\eta\|^2 \text{ for any } \eta\in L^2(\Omega),\, x\in\overline{\Omega}_1,\,\gamma\in \mathbb{N}.
\end{equation}
As a result,  $G(\eta)$ is analytic in $\Omega_1$ for any $\eta\in L^2(\Omega)$. In addition,
\begin{equation*}
G(\eta)\in C(\overline{\Omega}_1)
\end{equation*}
is uniformly Lipschitz with respect to $\eta\in\hat{\cB}$. Thus, for any fixed $\eta\in L^2(\Omega)$ and $x\in \overline{\Omega}_1$, $(G(\eta))(x)$ is well-defined.

\item Let $x\in\overline{\Omega}_1$ be fixed and the  functional $\bar{h}_x:\hat{\cB}\mapsto \mathbb{R}$ be defined by
\begin{equation*}
\bar{h}_x(\eta):=\bigr(G(\eta)\bigl)(x) \text{ for any }\eta\in \hat{\cB}.
\end{equation*}
Then,  $\bar{h}_x$ is uniformly Lipschitz with respect to $x\in\overline{\Omega}_1$ and hence $\bar{h}_x\in C_b(\hat{\cB})$.

\item  Let  $h\in \cH(\hat{\cB})$. The function $\overline{\Omega}_1\mapsto\mathbb{R}$ defined by
\begin{equation*}\label{definition-of-h}
H_h(x):=\langle h,\, \bar{h}_x\rangle_{C_b(\hat{\cB})^*, C_b(\hat{\cB})}=\left\langle h,\,\bigr(G(\cdot)\bigl)(x)\right\rangle_{C_b(\hat{\cB})^*, C_b(\hat{\cB})}\quad\text{for any } x\in \overline{\Omega}_1
\end{equation*}
 is well-defined.
Furthermore,
\begin{equation}\label{1.36}
H_h\in C(\overline{\Omega}_1).
\end{equation}

\end{enumerate}
\end{lemma}
Before giving the proof of Lemma \ref{lemma6}, we recall a result on analyticity of solutions to degenerate parabolic evolutions and applications (see \cite{liu2023observability}).
\begin{lemma}
\label{lemma-analytic}
Set $\alpha\in(0,2)$. Let $\omega_0$ be a subdomain of $\O$ with $0\not\in \bar{\omega}_0$. Then there are positive constants $C=C(\alpha,\O,\omega_0)\geq1$ and $\rho=\rho(\alpha,\O,\omega_0)$, $0<\rho\leq1$, such that when $x\in\omega_0$, $\forall\,0\leq s< t$, $\theta\in\N$ and $\gamma\in\N$, we have
\begin{equation*}
|\partial_x^\gamma\partial_t^\theta v(x,t)|\leq \dfrac{Ce^{\frac{C}{t-s}}\theta!\gamma!}{\rho^\gamma((t-s)/2)^\theta}\|v(\cdot,s)\|,
\end{equation*}
where $v$ solves the following forward equation:
\begin{equation*}
\left\{
\begin{array}{ll}
v_t(x,t)-Av(x,t) = 0, & \left( x ,t\right) \in \O\times(0,T),   \\[2mm]
v(1,t)=BC_\alpha\big(u(\cdot,t)\big) = 0, & t\in \left(0,T\right), \\[2mm]
v\left(x, 0\right) =v_{0}(x), &  x\in \O\,,
\end{array}
\right.
\end{equation*}
where $A$ is defined in (\ref{operator}).
\end{lemma}

Now, we are ready to prove Lemma \ref{lemma6}.
\begin{proof}[{Proof of Lemma \ref{lemma6}}]  

(i).  Consider the following backward equation:
\begin{equation}\label{1.24}
\left\{
\begin{array}{ll}
u_t(x,t)+Au(x,t) = 0, & \left( x ,t\right) \in \O\times(0,T),   \\[2mm]
u(1,t)=BC_\alpha\big(u(\cdot,t)\big) = 0, & t\in \left(0,T\right), \\[2mm]
u\left(x, T\right) =\eta, &  x\in \O\,.
\end{array}
\right.
\end{equation}
Let $\o_0=\O_1$, by Lemma \ref{lemma-analytic}, there are positive constants $C=C(\alpha,\O,\O_1)\geq1$ and $\rho=\rho(\alpha,\O,\O_1)$, $0<\rho\leq1$, such that the solution of (\ref{1.24}) satisfies
\begin{equation}\label{analytic-1}
|\partial_x^\gamma\partial_t^\theta u(x,t)|\leq \dfrac{Ce^{\frac{C}{s-t}}\theta!\gamma!}{\rho^\gamma((s-t)/2)^\theta}\|u(\cdot,s)\|,\,\forall\,\gamma,\theta\in\N,~\text{and}~t<s\leq T.
\end{equation}
In our case, for any $\eta\in L^2(\Omega)$,  consider the solution of adjoint equation: $\varphi(\cdot; \eta)$ (see \eqref{adjoint-state}).
By setting
\begin{equation}\label{1.26.1}
\hat u(t)=e^{-T\|a\|_{L^\infty(\Omega)}t}\varphi(T-t;\eta) \text{ for any }t\in[\tau,T],
\end{equation}
we find that it solves \eqref{1.24} with $A$ replaced by $\bar{A}$ defined as:
 $$
 \bar{A}=-A+T(\|a\|_{L^\infty(\Omega)}+a).
$$
By setting $\theta=0$ and $s=T$ in (\ref{analytic-1}), we obtain
$$
|\partial_x^\gamma u(x,t)|\leq \dfrac{Ce^{\frac{C}{T-t}}\gamma!}{\rho^\gamma}\|\eta\|,\,\forall\,\gamma\in\N,~\text{ for any }~t\in[\tau,T).
$$
This, together with \eqref{1.26.1}, implies  that
 \begin{equation*}
|\partial ^\gamma_x\varphi(x,T-t;\eta)|\leq e^{T\|a\|_{L^\infty(\Omega)}t}\dfrac{Ce^{\frac{C}{T-t}}\gamma!}{\rho^\gamma}\|\eta\|~\text{ for any }~t\in[\tau,T).
\end{equation*}
Replaced $T-t$ by $\bar{t}$, then 
denoted $t$ as $\bar{t}$ again, the above inequality implies  that
 \begin{equation*}\label{1.27}
 \begin{split}
|\partial ^\gamma_x\varphi(x,t;\eta)|&\leq e^{T(T-t)\|a\|_{L^\infty(\Omega)}}\dfrac{Ce^{\frac{C}{t}}\gamma!}{\rho^\gamma}\|\eta\|\\
&\leq e^{T^2\|a\|_{L^\infty(\Omega)}}\dfrac{CC(\tau,T)\gamma!}{\rho^\gamma}\|\eta\|~\text{ for any }~t\in[\tau,T).
 \end{split}
\end{equation*}
Denote by
$$
C_1=e^{T^2\|a\|_{L^\infty(\Omega)}}CC(\tau, T).
$$
We could derive that
 \begin{equation}\label{1.28}
|\partial ^\gamma_x\varphi(x,t;\eta)|\leq C_1\rho^{-\gamma}\gamma! \|\eta\| \text{ for }t\in[\tau,T),\, \eta\in L^2(\Omega).
\end{equation}
Recall the definition of $G(\eta)$ in  \eqref{definition-of-G}.
It holds that
$$
\bigr(\partial ^\gamma_x G(\eta)\bigl)(x)=\int^{T}_\tau\sum\limits_{\gamma_1+\gamma_2=\gamma}\left(\begin{matrix}
\gamma  \\
\gamma_1 
\end{matrix}
\right)\partial ^{\gamma_1}_x \varphi(x,t;\eta)\partial ^{\gamma_2}_x \varphi(x,t;\eta) dt,
$$
which, together with \eqref{1.28}, implies that
$$
\begin{array}{rl}
\left\vert\bigr(\partial ^\gamma_x G(\eta)\bigl)(x)\right\vert\leq&
\int^{T}_\tau\big|\sum\limits_{\gamma_1+\gamma_2=\gamma}\left(\begin{matrix}
\gamma  \\
\gamma_1 
\end{matrix}
\right)\partial ^{\gamma_1}_x \varphi(x,t;\eta)\partial ^{\gamma_2}_x \varphi(x,t;\eta)\big| dt\\
 \leq &T\sup\limits_{t\in[\tau,T)}\sum\limits_{\gamma_1+\gamma_2=\gamma}\left(\begin{matrix}
\gamma  \\
\gamma_1 
\end{matrix}
\right)\left|\partial ^{\gamma_1}_x \varphi(x,t;\eta)\right|\cdot\left|\partial ^{\gamma_2}_x \varphi(x,t;\eta)\right|\\
\leq & T\sum\limits_{\gamma_1+\gamma_2=\gamma}\left(\begin{matrix}
\gamma  \\
\gamma_1 
\end{matrix}
\right)C_1\rho^{-\gamma_1}\gamma_1!C_1\rho^{-\gamma_2}\gamma_2! \|\eta\|^2\\
\leq& TC_1^2\rho^{-\gamma}\gamma!  \|\eta\|^2 \text{ whenever } \eta\in L^2(\Omega),\, x\in\overline{\Omega}_1.
\end{array}
$$
This  advices  \eqref{key-inequality}.
As a result,   $G(\eta)$ is analytic in $\Omega_1$.

Let $x\in\overline{\Omega}_1$. Notice that for any $\hat x,\,x\in\overline{\Omega}_1$ and $\eta\in \hat{\cB}$, there is $\mu\in[0,1]$ so that
$$
\left[G(\eta)\right](\hat x)-\left[G(\eta)\right]( x)=\Bigr\langle \partial_x G(\eta)( x+\mu(\hat x- x)) ,\,\hat x-x\Bigl\rangle_{\mathbb{R}}.
$$
It follows from the uniform boundedness of $\partial_x G(\eta)$ (see \eqref{key-inequality} with $\gamma=1$) that there is  $C$ which is
independent of $\eta\in\hat{\cB}$ such that
\begin{equation}\label{1-45}
\Bigr|\left[G(\eta)\right](\hat x)-\left[G(\eta)\right]( x)\Bigl|
\leq C|\hat x-x|,\qquad
 \forall\,\eta\in \hat{\cB},\,\hat x,\,x\in\overline{\Omega}_1.
\end{equation}
Consequently, $G(\eta)$ exhibits uniform Lipschitz continuity with respect to $\eta\in \hat{\cB}$. Given that $G(\eta)\in C(\overline{\Omega}_1)$ holds for any $\eta\in \hat{\cB}$, we can assert that $(G(\eta))(x)$ is well-defined for all $x\in\overline{\Omega}_1$.

(ii).  Let $x\in\overline{\Omega}_1$ be fixed.
From the definition of $G$, it is derived   from  the  Cauchy-Schwarz inequality that
\begin{equation*}
\begin{split}
\left|\left[G(\hat \eta)-G(\eta)\right](x)\right|^2&=\left[\int^{T}_\tau\left|\varphi(x,s;\hat \eta)+\varphi(x,s;\eta)\right|\cdot \left|\varphi(x,s;\hat \eta)-\varphi(x,s;\eta)\right| ds\right]^2\\
&\leq 2\left[\int^{T}_\tau|\varphi(x,s;\hat \eta)|^2+|\varphi(x,s;\eta)|^2 ds\right]\cdot\int^{T}_\tau|\varphi(x,s;\hat \eta-\eta)|^2 ds\\
&=2\left[\left(G(\hat \eta)\right)(x)+\left(G(\eta)\right)(x)\right]\cdot \left[G(\hat \eta-\eta)\right](x),\quad \forall\, \hat \eta,\eta\in\hat{\cB}.
\end{split}
\end{equation*}
From \eqref{key-inequality} with $\gamma=0$,
  there is $C>0$ such  that
\begin{equation}\label{1.53}
\left|\bar{h}_x(\hat \eta)-\bar{h}_x(\eta)\right|=\left|\left[G(\hat \eta)-G(\eta)\right](x)\right|\le C\|\hat \eta-\eta\|,\quad \forall\, \hat \eta,\eta\in\hat{\cB},
x\in\overline{\Omega}_1,
\end{equation}
which shows that  $G(\cdot)(x)$ is uniformly Lipschitz in  $\hat{\cB}$ and
 $\bar{h}_x\in C_b(\hat{\cB})$.

(iii). Let  $h\in \cH(\hat{\cB})$. Because $\bar{h}_x\in C_b(\hat{\cB})$ for any fixed $x\in\overline{\Omega}_1$,
the function defined by
\begin{equation}\label{definition-of-H}
H_h(x):=h(\bar{h}_x)=h\Bigr(\bigr(G(\cdot)\bigl)(x)\Bigl)\qquad\text{for any } x\in \overline{\Omega}_1
\end{equation}
 is well-defined. Besides, it follows from \eqref{1-45} that
 $$
 \Bigr\|\left[G(\cdot)\right](\hat x)-\left[G(\cdot)\right]( x)\Bigl\|_{C_b(\hat{\cB})}
\leq C|\hat x-x|,
$$
which gives  \eqref{1.36}. The proof is completed.
\end{proof}

\begin{lemma}\label{lemma9} Let $\beta\in\cB$ and $h \in  \cH(\hat{\cB})$. Then, the following equality holds
\begin{equation}\label{1.38}
h\biggr(\int_{\Omega_1}\beta(x)\bigr(G(\cdot)\bigl)(x) dx\biggl)=\int_{\Omega_1}\beta(x)h\Bigr(\bigr(G(\cdot)\bigl)(x)\Bigl) dx.
\end{equation}
\end{lemma}

\begin{proof}
The  proof will be accomplished  through  the following steps.

{\it  Step 1: Show  that \eqref{1.38} with $\beta=\chi_E$ for any measurable set $E\subseteq \Omega_1$. } 

For any given $\varepsilon>0$, we can utilize the uniform Lipschitz continuity of $G(\eta)$, as stated in the first item of Lemma \ref{lemma6}, to deduce the there is $\delta>0$ such that the following inequality holds:
\begin{equation}\label{1-48}
\bigr|\left[G(\eta)\right](\hat x)-\left[G(\eta)\right]( x)\bigl|\leq\varepsilon \quad\text{ for all }\eta\in \hat{\cB},\,\hat x,\,x\in\overline{\Omega}_1,\, |\hat x-x|\leq\delta.
\end{equation}
Notice that
$$
\overline{\Omega}_1\subseteq \bigcup\limits_{x\in \overline{\Omega}_1}\left\{\hat x\in\overline{\Omega}_1\bigm| |\hat x-x|<\delta\right\}.
$$
It follow from the compactness of $\overline{\Omega}_1$ that there is a partition of $\overline{\Omega}_1$, $\{E_{\varepsilon, 1}, E_{\varepsilon, 2},\cdots, E_{\varepsilon, n_\varepsilon}\}$ so that
for any  $j\in\{1,2,\cdots, n_\varepsilon\}$, there  is $x_j\in \overline{\Omega}_1$ satisfying
$$
E_{\varepsilon, j}\subseteq \left\{\hat x\in\overline{\Omega}_1\bigm| |\hat x-x_j|\leq\delta\right\}.
$$
Then, it is derived from \eqref{1-48} that
\begin{equation}\label{1-49}
\left|\int_{\Omega_1}\chi_E(x)\bigr(G(\eta)\bigl)(x) dx-\sum\limits_{j=1}^{n_\varepsilon} \bigr(G(\eta)\bigl)(x_j)|E\cap E_{\varepsilon,j}|\right|\leq \varepsilon \cdot |\Omega_1|
 \text{ for any }\eta\in \hat{\cB}.
\end{equation}
From the uniform Lipschitz property of $\bar{h}_x$ in the second item of Lemma \ref{lemma6}, we have
\begin{equation*}\label{1-50}
\int_{\Omega_1}\chi_E(x)\bigr(G(\cdot)\bigl)(x) dx\in C_b(\hat{\cB}), \qquad \bigr(G(\cdot)\bigl)(x_j)\in C_b(\hat{\cB}), \forall  j=1,\cdots,n_\varepsilon.
\end{equation*}
This, together with \eqref{1-49}, implies that
$$
\lim\limits_{\varepsilon\rightarrow0+}\left\|\int_{\Omega_1}\chi_E(x)\bigr(G(\cdot)\bigl)(x) dx-\sum\limits_{j=1}^{n_\varepsilon} \bigr(G(\cdot)\bigl)(x_j)|E\cap E_j|\right\|_{C_b(\hat{\cB})}=0.
$$
Since  $h\in C_b(\hat{\cB})^*$,  and
\begin{align*}
&\left|\int_{\Omega_1}\chi_E(x)h\Bigr(\bigr(G(\cdot)\bigl)(x)\Bigl) dx-
\sum\limits_{j=1}^{n_\varepsilon}h\Bigr( \bigr(G(\cdot)\bigl)(x_j)\Bigl)|E\cap E_j|\right|\\
=&\left|\sum\limits_{j=1}^{n_\varepsilon}\int_{E\cap E_j}
\left[h\Bigr(\bigr(G(\cdot)\bigl)(x)\Bigl)-
h\Bigr(\bigr(G(\cdot)\bigl)(x_j)\Bigl)\right]
dx\right|
\leq  \sum\limits_{j=1}^{n_\varepsilon}\int_{E\cap E_j}
\|h\|\varepsilon
 dx,
\end{align*}
which shows
\begin{equation}\label{bz5}
\begin{split}
&h\biggr(\int_{\Omega_1}\chi_E(x)\bigr(G(\cdot)\bigl)(x) dx\biggl)
=\lim\limits_{\varepsilon\rightarrow0+}h\biggr(\sum\limits_{j=1}^{n_\varepsilon} \bigr(G(\cdot)\bigl)(x_j)|E\cap E_j|\biggl)\\
=&\lim\limits_{\varepsilon\rightarrow0+}\sum\limits_{j=1}^{n_\varepsilon}h\Bigr( \bigr(G(\cdot)\bigl)(x_j)\Bigl)|E\cap E_j|
=\int_{\Omega_1}\chi_E(x)h\Bigr(\bigr(G(\cdot)\bigl)(x)\Bigl)dx.
\end{split}
\end{equation}

{\it  Step 2: Show \eqref{1.38}.}
 
 From (\ref{bz5}),  we find that  \eqref{1.38} holds whenever  $\beta$ is a  simple function, that is,
\begin{equation}\label{1-52}
h\biggr(\int_{\Omega_1}\beta(x)\bigr(G(\cdot)\bigl)(x) dx\biggl)=\int_{\Omega_1}\beta(x)h\Bigr(\bigr(G(\cdot)\bigl)(x)\Bigl) dx
 \text{ whenever  }\beta=\sum\limits_{k=1}^{m} a_k\chi_{E_k}.
\end{equation}
For any $\beta\in\cB$, there is a sequence of simple functions $\{\beta_n,n\in\mathbb{N}\}$  so that
$$
\lim\limits_{n\rightarrow\infty}\beta_n=\beta\quad\text{strongly in } L^1(\Omega).
$$
From \eqref{1.53}, we have
$$
\int_{\Omega_1}\beta(x)\bigr(G(\cdot)\bigl)(x) dx\in C_b(\hat{\cB}),\qquad
\int_{\Omega_1}\beta_n(x)\bigr(G(\cdot)\bigl)(x)dx\in C_b(\hat{\cB}), n=1,2,\cdots.
$$
Notice that
\begin{equation*}
\begin{split}
&\left\|\int_{\Omega_1}\left[\beta_n(x)-\beta(x)\right]\bigr(G(\cdot)\bigl)(x) dx\right\|_{C(\hat{\cB})}
=\sup\limits_{\eta\in\hat{\cB}}\left|\int_{\Omega_1}\left[\beta_n(x)-\beta(x)\right]\bigr(G(\eta)\bigl)(x)dx\right|\\
\leq&\int_{\Omega_1}\left|\beta_n(x)-\beta(x)\right|dx\cdot \sup\limits_{\eta\in\hat{\cB}}\|G(\eta)\|_{L^\infty(\Omega_1)}
=\int_{\Omega_1}\left|\beta_n(x)-\beta(x)\right|dx\cdot \sup\limits_{\eta\in\hat{\cB}}\|G(\eta)\|_{C(\Omega_1)},
\end{split}
\end{equation*}
and
$\{G(\eta)\mid \eta\in\hat{\cB}\}$ is uniformly bounded in $C(\overline{\Omega}_1)$ from \eqref{key-inequality} with $\gamma=0$. We then have
$$
\lim\limits_{n\rightarrow\infty}h\biggr(\int_{\Omega_1}\beta_n(x)\bigr(G(\cdot)\bigl)(x) dx\biggl)=h\biggr(\int_{\Omega_1}\beta(x)\bigr(G(\cdot)\bigl)(x) dx\biggl).
$$
On the other hand, it follows from \eqref{1.36} that
$$
\lim\limits_{n\rightarrow\infty}\int_{\Omega_1}\beta_n(x)h\Bigr(\bigr(G(\cdot)\bigl)(x)\Bigl) dx=\int_{\Omega_1}\beta(x)h\Bigr(\bigr(G(\cdot)\bigl)(x)\Bigl) dx.
$$
The  \eqref{1.38} then follows directly from \eqref{1-52}.
\end{proof}

\begin{lemma}\label{lemma10}
For any $h\in\cH(\hat{\cB})$, $H_h$   defined by \eqref{definition-of-H}  is analytic in $\Omega_1$.
\end{lemma}
\begin{proof} Let $x\in\Omega_1$.
We assert that the equality
\begin{equation}\label{definition-of-derivation}
\partial ^\gamma H_h(x)=\partial ^\gamma h\left(\bigl(G(\cdot)\bigr)(x)\right)= h\left(\left(\partial ^\gamma G(\cdot)\right)(x)\right)\quad\text{for all }\gamma\in \mathbb{N},
\end{equation}
holds, and we will prove it using mathematical induction.
It is evident that \eqref{definition-of-derivation} is valid when $\gamma=0$.
Now, assuming that \eqref{definition-of-derivation} holds for any $\gamma\leq k$, let us consider the case of $k+1$. 
Take
$x_1\in\O_1$ such that 
$$
x+\varepsilon x_1\in{\Omega_1}\qquad \text{when }\varepsilon>0\text{ small enough}.
$$
Let $\eta\in\hat{\cB}$. Since  $G(\eta)$ is analytic,
by differential mean value theorem,
\begin{equation*}
\begin{split}
&\left|\frac{1}{\varepsilon}\left[ \partial ^k\bigr(G(\eta)\bigl)(x+\varepsilon x_1)-\partial ^k \bigr(G(\eta)\bigl)(x)\right]-\partial ^{k+1} \bigr(G(\eta)\bigl)(x)\right|\\
\leq&\frac{\varepsilon}{2}\max\limits_{\mu\in [0,1]}\left|\partial ^{k+2} \bigr(G(\eta)\bigl)(x+\mu\varepsilon x_1)\right|.
\end{split}
\end{equation*}
By  \eqref{key-inequality}, it follows that
\begin{equation}\label{1.62}
\lim\limits_{\varepsilon\rightarrow0+}\left[
\frac{1}{\varepsilon}\left[ \partial ^k \bigr(G(\cdot)\bigl)(x+\varepsilon x_1)-\partial ^\gamma \bigr(G(\cdot)\bigl)(x)\right]-\partial ^{k+1} \bigr(G(\cdot)\bigl)(x)\right]=0
\quad\text{strongly in } C_b(\bar{\cB}).
\end{equation}

On the other hand, notice that
\begin{equation*}
	\begin{split}
&\left|\left[\partial ^{k+1} \bigr(G(\hat \eta)\bigl)-\partial ^{k+1} \bigr(G(\eta)\bigl)\right](x)\right|\\
\leq&  \int^{T}_\tau\left|\partial ^{k+1} \varphi^2(x,s;\hat \eta)-\partial ^{k+1} \varphi^2(x,s; \eta)\right|
ds\\
\leq&\sum\limits_{\gamma_1+\gamma_2= k+1} \int^{T}_\tau\left|\partial^{\gamma_1}\varphi(x,s;\hat \eta)\partial^{\gamma_2} \varphi(x,s;\hat \eta)-\partial^{\gamma_1}\varphi(x,s; \eta)\partial^{\gamma_2} \varphi(x,s; \eta)\right| ds\\
=& \sum\limits_{\gamma_1+\gamma_2=k+1} \int^{T}_\tau\left|\partial^{\gamma_1}\varphi(x,s;\hat \eta)\partial^{\gamma_2} \varphi(x,s;\hat \eta- \eta)+\partial^{\gamma_1}\varphi(x,s; \hat \eta-\eta)\partial^{\gamma_2} \varphi(x,s; \eta)\right| ds\\
\leq& 2\sum\limits_{\gamma_1+\gamma_2=k+1}\sqrt{ \int^{T}_\tau\left[\partial^{\gamma_1}\varphi(x,s;\hat \eta)\right]^2 ds}\sqrt{\int^{T}_\tau\left[\partial^{\gamma_2} \varphi(x,s;\hat \eta- \eta)\right]^2 ds}
\end{split}
\end{equation*}
for any $\hat \eta,\eta\in\hat{\cB}$.
 It follows from \eqref{1.28} that
 $$
 \partial ^{k+1} \bigr(G(\cdot)\bigl)(x)\in  C_b(\hat{\cB}).
$$
This, together with \eqref{1.62}, implies that
$$
\lim\limits_{\varepsilon\rightarrow0+}
\frac{1}{\varepsilon}\left[ \partial ^k\bigr(G(\cdot)\bigl)(x+\varepsilon x_1)-\partial ^k \bigr(G(\cdot)\bigl)(x)\right]=\partial ^{k+1} \bigr(G(\cdot)\bigl)(x)
\quad\text{strongly in } C_b(\hat{\cB}).
$$
Furthermore,
$$
\lim\limits_{\varepsilon\rightarrow0+}
h\left(\frac{1}{\varepsilon}\left[\bigr(\partial ^k G(\cdot)\bigl)(x+\varepsilon x_1)-\bigr(\partial ^k G(\cdot)\bigl)(x)\right]\right)
=
h\left(\partial ^{k+1} \bigr(G(\cdot)\bigl)(x)\right).
$$
Since
\begin{equation*}\label{1-54}
	\begin{split}
\partial ^{k+1} H_h(x)
&=\partial \bigr(\partial ^{k} H_h(x)\bigl)=\partial \Bigr(h\bigr(\bigr(\partial ^k G(\cdot)\bigl)(x)\bigl)\Bigl)\\
&=\lim\limits_{\varepsilon\rightarrow0+}
\frac{1}{\varepsilon}\left[h\Bigr(\bigr(\partial ^k G(\cdot)\bigl)(x+\varepsilon x_1)\Bigl)-h\Bigr(\bigr(\partial ^k G(\cdot)\bigl)(x)\Bigl)\right]\\
&=\lim\limits_{\varepsilon\rightarrow0+}
h\left(\frac{1}{\varepsilon}\left[\bigr(\partial ^k G(\cdot)\bigl)(x+\varepsilon x_1)-\bigr(\partial ^k G(\cdot)\bigl)(x)\right]\right),
\end{split}
\end{equation*}
the function $\partial ^{k} H_h$ is  differentiable, and \eqref{definition-of-derivation} holds when $k+1$. Hence, the claim \eqref{definition-of-derivation} is valid. Now, by utilizing \eqref{key-inequality} and (\ref{definition of B}), we obtain
$$
\left\vert\bigr  (\partial ^\gamma G(\eta)\bigl)(x)\right\vert\leq TC_1^2\rho^{-\gamma}\gamma!  \delta_0^2\qquad\text{for any } \eta\in \hat{\cB},\, x\in{\Omega_1},\,\gamma\in \mathbb{N},
$$
where $\delta_0$ is defined in \eqref{definition-of-delta}.
This, together with  \eqref{definition-of-derivation}, implies that $H_h$ is analytic in $\Omega_1$. The proof is completed.
\end{proof}

The following Lemma \ref{lemma11} is  well known  (see, e.g., \cite{privat2015complexity} or \cite{guo2015optimal}).
  \begin{lemma}\label{lemma11}
Let $\phi(\cdot)$ be analytic in $\Omega$. Then, the problem
$$
   \sup\limits_{ \beta\in\cB}\displaystyle\int_\Omega
 \beta(x)\phi(x) dx
$$
 admits a unique solution $\bar\beta$, and there is $\bar \omega\in W$ such that
\begin{equation*}
\bar\beta=\chi_{\bar \omega}.
\end{equation*}
Furthermore, $\bar \omega$ is an upper  level set of  $\phi$, i.e. there is $c\in\mathbb{R}$ so that
$$
\bar\omega=\left\{x\in \Omega\bigm| \phi(x)\geq c\right\}.
$$
  \end{lemma}

Now, we are ready to prove Theorem \ref{main-theorem}.

 \begin{proof}[{Proof of Theorem \ref{main-theorem}}]
 According to Corollary \ref{remark1}, it can be observed that {\bf (ROP)} possesses at least one solution, denoted by $\beta^*$. Based on Proposition \ref{Neumann}, we can conclude that there exists $h^*\in \cH(\hat{\cB})$ that solves {\bf (SP4)}, while $\beta^*$ solves the optimization problem given by
$$
\sup\limits_{\beta\in\cB}\widetilde J(\beta,h^*).
$$
This, together with the definition of $\widetilde J(\cdot,\cdot)$ and Lemma \ref{lemma9}, implies that
$\beta^*$ solves
\begin{equation*}
\sup\limits_{\beta\in\cB}\int_{\Omega_1}\beta(x)H_{h^*}(x) dx=\sup\limits_{\beta\in\cB}\int_{\Omega_1}\beta(x)h^*\Bigr(\bigr(G(\cdot)\bigl)(x)\Bigl)dx.
\end{equation*}

Based on Lemma \ref{lemma10}, we establish that $H_{h^*}$ possesses analyticity. By applying Lemma \ref{lemma11}, we can conclude that there exists $\omega^*\in W$ such that
\begin{equation*}\label{1-56}
\beta^*=\chi_{\omega^*}.
\end{equation*}
This observation indicates that the relaxed solution to {\bf (ROP)} also serves as a solution to {\bf (OP)}, thereby demonstrating that {\bf (OP)} possesses at least one solution.
Moreover, employing Lemma \ref{lemma11} once more, we ascertain that $\beta^*$ must represent an upper-level set of some analytic function. Hence, the proof is concluded.
\end{proof}

\section{Declarations}
The authors have not disclosed any competing interests and data availability.

\section{Acknowledgement}

This work was supported by the National Natural Science Foundation of China, the Science Technology Foundation of Hunan Province.

\bibliographystyle{abbrvnat}
\bibliography{ref.bib}
\end{document}